\documentclass[10pt,reqno,a4paper,final]{amsart}%
\usepackage[a4paper,left=2.8cm,right=2.7cm,top=2.7cm,bottom=3.5cm]{geometry}%
\usepackage[pagebackref,unicode=true]{hyperref}%
\usepackage[american]{babel}%
\usepackage{marginnote,xcolor,mathtools,graphicx,lmodern}%
\graphicspath{{figs/}}%
\usepackage[utf8]{inputenc}%
\usepackage{enumerate}%
\usepackage{tikz}%
\usetikzlibrary{angles,quotes}%
\allowdisplaybreaks%
\numberwithin{equation}{section}%
\title[Riesz equilibrium problem and integral identities for special functions]{On the solution of a Riesz equilibrium problem\\and integral identities for special functions}
\author{Djalil Chafaï}%
\address[DC]{CEREMADE, Université Paris-Dauphine, Université PSL, CNRS, 75016 Paris, France}%
\email{\url{mailto:djalil(at)chafai.net}}%
\urladdr{\url{http://djalil.chafai.net/}}%
\author{Edward B.\ Saff}%
\address[ES]{Center for Constructive Approximation, %
  Vanderbilt University, Nashville, TN 37240, USA}%
\email{\url{mailto:Ed.Saff@Vanderbilt.Edu}}%
\urladdr{\url{https://my.vanderbilt.edu/edsaff/}}%
\author{Robert S.\ Womersley}%
\address[RW]{School of Mathematics and Statistics, %
  University of New South Wales, Sydney {\tiny NSW 2052},
  Australia}%
\email{\url{mailto:R.Womersley@unsw.edu.au}}%
\urladdr{\url{https://web.maths.unsw.edu.au/~rsw/}}%
\date{Summer 2021, revised Autumn 2021, revised Spring 2022, published in J.
  Math. Anal. Appl. 515 (2022) 126367 DOI:
  \href{https://10.1016/j.jmaa.2022.126367}{10.1016/j.jmaa.2022.126367}.
  The present document is an arXiv post-publication update dated Autumn 2022,
  with an appendix giving an analytic proof of the Riesz formula adapted
  from Dyda-Kuznetsov-Kwa\'{s}nicki 2017. Compiled \today}
\keywords{Potential theory; Equilibrium measure; Riesz kernel; Arcsine
    distribution; Euler--Lagrange (Frostman) conditions; Funk--Hecke formula;
    Special functions and integral identities; Hypergeometric function;
    Elliptic integral; Landen transform}
\subjclass[2000]{%
  31A10; 
  31B10; 
  44A20; 
  33C20; 
  33C75
}
\newtheorem{theorem}{Theorem}[section]%
\newtheorem{lemma}[theorem]{Lemma}%
\newtheorem{remark}[theorem]{Remark}%
\newtheorem{corollary}[theorem]{Corollary}%
\makeatletter
\def\@MRExtract#1 #2!{#1}     
\renewcommand{\MR}[1]{
  \xdef\@MRSTRIP{\@MRExtract#1 !}%
  \href{http://www.ams.org/mathscinet-getitem?mr=\@MRSTRIP}{MR-\@MRSTRIP}}
\makeatother
\begin{document}
\begin{abstract}
  The aim of this note is to provide a full space quadratic external field
  extension of a classical result of Marcel Riesz for the equilibrium measure
  on a ball with respect to Riesz $s$-kernels. We address the case $s = d - 3$
  for arbitrary dimension $d$, in particular the logarithmic kernel in
  dimension $3$. The equilibrium measure for this full space external field
  problem turns out to be a radial arcsine distribution supported on a ball
  with a special radius. As a corollary, we obtain new integral identities
  involving special functions such as elliptic integrals and more generally
  hypergeometric functions. It seems that these identities are not found in
  the existing tables for series and integrals, and are not recognized by
  advanced mathematical software. Among other ingredients, our proofs involve
  the Euler\,--\,Lagrange variational characterization, the Funk\,--\,Hecke
  formula, the Weyl regularity lemma, the maximum principle, and special
  properties of hypergeometric functions.
\end{abstract}
\maketitle

{\footnotesize\tableofcontents}

\section{Introduction and main results}

The goal of this note is to provide a full space quadratic external field
extension (Theorem \ref{th:d=s+3} below) of a classical result of Marcel Riesz
(Theorem \ref{th:riesz} below) for the equilibrium measure on a ball in
arbitrary dimensions with respect to Riesz $s$-kernels, including the
logarithmic kernel. The equilibrium measure turns out to be a radial arcsine
distribution. As corollaries, we obtain new integral identities involving
special functions such as elliptic integrals and more generally hypergeometric
functions; see, for example, Corollaries \ref{co:riesz}, \ref{co:formulas},
and \ref{co:formulas:more} below. It seems that these identities are not found
in the existing tables for series and integrals, and are not recognized by
advanced mathematical software.

Before we present our results and identities, we recall some basic notions
from potential theory. Throughout this note, we denote by $d$ the Euclidean
dimension, which is always a positive integer, and by $s\in(-2,+\infty)$ the
Riesz parameter. For $x\in\mathbb{R}^d$, $x\neq 0$, the Riesz $s$-kernel is
defined by
\begin{equation}\label{eq:VW}
  K_s(x):=
  \begin{cases}
    \mathrm{sign}(s)\left|x\right|^{-s} & \text{if $-2<s<0$ or $s>0$}\\
    -\log\left|x\right| & \text{if $s=0$}
  \end{cases},
\end{equation}
where $\left|x\right|:=\sqrt{x_1^2+\cdots+x_d^2}$ is the Euclidean norm. It is
the Coulomb or Newton kernel if $s=d-2$. Let $\mathcal{M}_1$ be the set of
probability measures on $\mathbb{R}^d$ and let
$V:\mathbb{R}^d\mapsto(-\infty,+\infty]$ be a lower semicontinuous function,
which will play the role of an external field. In this note we only deal with
either an external field constant on a centered ball and infinite outside the
ball, or with a quadratic external field of the form
$V(\cdot)=\gamma\left|\cdot\right|^2$, $\gamma>0$. The energy of
$\mu\in\mathcal{M}_1$ with external field $V$ is defined by
\begin{equation}\label{eq:EsV}
  \mathrm{I}(\mu)
  := \iint_{\mathbb{R}^d\times\mathbb{R}^d}(K_s(x-y)+V(x)+V(y))\mu(\mathrm{d} x)\mu(\mathrm{d} y)
  \in(-\infty,+\infty].
\end{equation}
For $s\in(-2,d)$, with our choices of $V$, the integrand in the double
integral in \eqref{eq:EsV} is bounded below, $\mathrm{I}$ is strictly
convex\footnote{In other words, $K_s$ is conditionally strictly positive in the
  sense of Bochner, see for instance \cite[Section 4.4]{MR3970999}.} on
$\mathcal{M}_1$ and lower semicontinuous with compact level sets\footnote{We
  follow the probability theory standard and equip the convex set
  $\mathcal{M}_1$ with the topology of weak convergence with respect to
  continuous and bounded test functions, in other words the weak-$*$
  convergence.}. It has a unique global minimizer called the ``equilibrium
measure'' $\mu_{\mathrm{eq}}\in\mathcal{M}_1$; in other words,
\begin{equation}\label{eq:equ}
  \mathrm{I}(\mu_{\mathrm{eq}})= \min_{\mu\in\mathcal{M}_1}\mathrm{I}(\mu)
  \quad\text{and}\quad
  \mathrm{I}(\mu)>\mathrm{I}(\mu_{\mathrm{eq}})
  \text{ for all $\mu\neq\mu_{\mathrm{eq}}$, $\mu\in\mathcal{M}_1$}.
\end{equation}
Moreover, $\mu_{\mathrm{eq}}$ is compactly supported with finite energy
$\mathrm{I}(\mu_{\mathrm{eq}})<+\infty$. We refer to \cite{MR0350027} and \cite{MR3970999}
for more details. If $s<0$, then $K_s$ is not singular and we could have
$\mathrm{I}(\mu)<\infty$ for a $\mu\in\mathcal{M}_1$ having Dirac masses; in
particular $\mu_{\mathrm{eq}}$ could conceivably have Dirac masses. In contrast, if
$s\geq0$ then $K_s$ is singular and $\mathrm{I}(\mu)=+\infty$ whenever $\mu$
has Dirac masses; consequently $\mu_{\mathrm{eq}}$ cannot have such masses.

We first recall a classical result of M.~Riesz for the equilibrium measure
with constant external field in a closed ball and infinite outside the ball.
For $R>0$, let
\[
  B_R:=\{x\in\mathbb{R}^d:|x|\leq R\}
  \quad\text{and}\quad
  S_R:=\{x\in\mathbb{R}^d:|x|=R\}
\]
denote the ball and sphere of radius $R$ centered at the origin. In particular
$S_1=\mathbb{S}^{d-1}$ is the unit sphere, with surface area
$|\mathbb{S}^{d-1}|=2\pi^{d/2}/\Gamma(d/2)$. For a subset $S$ of
$\mathbb{R}^d$, we denote, when it makes sense, by $\sigma_{S}$ the uniform
probability measure on $S$ (normalized trace of Lebesgue measure).

\begin{theorem}[Riesz \cite{zbMATH03029943}]\label{th:riesz}
  Suppose that $d\in\{2,3,4,\ldots\}$ and $V=\begin{cases}
    0&\text{on $B_R$}\\
    +\infty&\text{outside $B_R$}\end{cases}$, $R>0$.
  \begin{itemize}
  \item If $-2<s\leq d-2$, then $\mu_{\mathrm{eq}}=\sigma_{S_R}$,
  \item If $d-2<s<d$, then $\mu_{\mathrm{eq}}$ is the probability measure
    \begin{equation}\label{eq:rieszmeq}
      \mu_{\mathrm{eq}}(\mathrm{d} x)=
      \frac{\Gamma(1+\frac{s}{2})}{R^s\pi^{\frac{d}{2}}\Gamma(1+\frac{s-d}{2})}
      \frac{\mathbf{1}_{|x|\leq R}}{(R^2-|x|^2)^{\frac{d-s}{2}}}\mathrm{d} x
      =\frac{2\Gamma(1+\frac{s}{2})}{R^s\Gamma(1+\frac{s-d}{2})\Gamma(\frac{d}{2})}
      \frac{r^{d-1}\mathbf{1}_{r\leq R}}{(R^2-r^2)^{\frac{d-s}{2}}}\,\mathrm{d}r
      \mathrm{d} \sigma_{S_1},
    \end{equation}
    where $\mathrm{d} x$ and $\mathrm{d} r$ denote the Lebesgue measures on $\mathbb{R}^d$
    and on $[0,+\infty)$ respectively.\\
    Moreover, the equilibrium potential $U^{\mu_{\mathrm{eq}}}$ satisfies, for $x\in B_R$,
  \begin{equation}\label{eq:rieszumeq}
    U^{\mu_{\mathrm{eq}}}(x)
    :=(K_s*\mu_{\mathrm{eq}})(x) = \int_{\mathbb{R}^d} K_s(x-y) \mu_{\mathrm{eq}}(\mathrm{d} y)
    =\mathrm{I}(\mu_{\mathrm{eq}})=
    \frac{\Gamma(1+\frac{s}{2})\Gamma(\frac{d-s}{2})}{R^s\Gamma(\frac{d}{2})}.
  \end{equation}
  \end{itemize}
\end{theorem}

The case $d-2<s<d$ in Theorem \ref{th:riesz} is a direct consequence of the
following formula.

\begin{lemma}[Riesz formula \cite{zbMATH03029943}]\label{le:riesz}
  If $d\in\{2,3,4,\ldots\}$, $0\leq d-2<s<d$, and $R>0$, then for $x\in B_R$,
  \begin{equation}\label{eq:riesz}
    \int_{\mathbb{R}^d}
    \frac{|x-y|^{-s}}{(R^2-|y|^2)^{\frac{d-s}{2}}}
    \mathbf{1}_{|y|\leq R}\mathrm{d} y
    =\frac{\pi^{\frac{d}{2}+1}}{\Gamma(\frac{d}{2})\sin(\frac{\pi}{2}(d-s))}.
  \end{equation}
\end{lemma}

The proof of Theorem \ref{th:riesz} and Lemma \ref{le:riesz} can be found,
together with some geometric aspects, in the works of M.~Riesz
\cite[p.~438--439]{zbMATH03003508} and \cite[\S~16, Eq.~(1)]{zbMATH03029943},
where it is mentioned that the cases $d=1,2,3$ were already considered by
Pólya and Szeg\H{o} in \cite{polya-szego}. It can also be found in the book
\cite[\S~II.3.13, p.~163--164, and Appendix p.~399\,--\,400]{MR0350027}, and is
stated in \cite[Eq.~(4.6.13)]{MR3970999}. The proof sketched by Riesz, with a
bit more detail by Landkof, involves first a geometric inversion 
transforming the integral on the ball into an integral on its complement, and
second a trigonometric substitution which has a geometric interpretation, both
steps being inspired by the analytic-geometric techniques used classically for
elliptic integrals since the eighteenth century. For the reader's convenience,
a detailed proof of Lemma \ref{le:riesz} is given in Appendix
\ref{ap:le:riesz}. After publication of the present article in Journal of Mathematical Analysis and Applications, we have found that Dyda, Kuznetsov, and M. Kwaśnicki gave in \cite{MR3640641} an alternative analytic (and short!) proof of Lemma \ref{le:riesz}. Since this proof is somewhat hidden by the Meijer G-function material in \cite{MR3640641}, we give, in Appendix \ref{ap:se:analytic:proof:riesz} of the present arXiv post-publication update, a straight version of this analytic proof, without using Meijer G-functions.

Our first result, Corollary \ref{co:riesz}, is a simple consequence of Theorem
\ref{th:riesz}. It relates an equilibrium measure of potential theory with an
integral identity for special functions (here a ${}_2F_1$ hypergeometric
function). Before stating it, let us recall the \emph{Newton binomial series}
\begin{equation}\label{eq:binom}
  \frac{1}{(1-z)^{\alpha}}=\sum_{n=0}^\infty(\alpha)_n\frac{z^n}{n!},
  \quad \alpha,z\in\mathbb{C},\ |z|<1,
\end{equation}
where $(\alpha)_n:=\alpha(\alpha+1)\cdots(\alpha+n-1)$ is the \emph{Pochhammer
  symbol} for the rising factorial, with the convention $(\alpha)_0:=1$ if
$\alpha\neq0$. If $\Re(\alpha)>0$, then
$(\alpha)_n=\Gamma(\alpha+n)/\Gamma(\alpha)$. More generally, the
\emph{hypergeometric function} with parameters
$(a_1,\ldots,a_p)\in\mathbb{C}^p$ and $(b_1,\ldots,b_q)\in\mathbb{C}^q$, at
$z\in\mathbb{C}$, $|z|<1$, is given (when it makes sense) by the series
\begin{equation}\label{eq:hype}
  {}_pF_q(a_1,\ldots,a_p;b_1,\ldots,b_q;z)
  :=\sum_{n=0}^\infty
  \frac{(a_1)_n\cdots(a_p)_n}{(b_1)_n\cdots(b_q)_n}
  \frac{z^n}{n!}.
\end{equation}
Special choices of the parameters $a_1,\ldots,a_p$ and $b_1,\ldots,b_q$ allow
us to recover many special functions, for instance
${}_2F_1(\alpha,\beta;\beta;z)=(1-z)^{-\alpha}$,
${}_2F_1(1,1;2;-z)=\frac{\log(1+z)}{z}$, and
${}_2F_1(\frac{1}{2},\frac{1}{2};\frac{3}{2};z^2)=\frac{\arcsin(z)}{z}$.
Actually one of the main historical motivations for the introduction and study
of hypergeometric functions is the unification of as many as possible special
functions via series expansions. For instance the complete \emph{elliptic
  integral} of first and second kind, $K$ and $E$ respectively, satisfy for $z\in[0,1]$,
\begin{align}
  K(z)
  :=
    \int_0^{\frac{\pi}{2}}\frac{\mathrm{d} \theta}{\sqrt{1-z\sin^2(\theta)}}
    =\int_0^1\frac{\mathrm{d} t}{\sqrt{1-zt^2}\sqrt{1-t^2}}
  =\frac{\pi}{2}\,{}_2F_1\Bigr(\frac{1}{2},\frac{1}{2};1;z\Bigr)
  \label{eq:EllipticK}
  \end{align}
  and
  \begin{align}
  E(z)
  :=\int_0^{\frac{\pi}{2}}\sqrt{1-z\sin^2(\theta)} \; \mathrm{d} \theta
   =\int_0^1\frac{\sqrt{1-zt^2}}{\sqrt{1-t^2}}\,\mathrm{d} t
  =\frac{\pi }{2}\,{}_2F_1\Bigr(-\frac{1}{2},\frac{1}{2};1;z\Bigr).\label{eq:EllipticE}
\end{align}
They can be extended to the complex plane, with a branch cut discontinuity
running from $1$ to $\infty$. For more basic facts about these functions, we
refer for instance to the classical books \cite{MR0099454,MR0277773}.
Here we only remark that $K,E\geq0$ on the interval $[0,1]$,
$K(0)=E(0)=\frac{\pi}{2}$, $K(1)=\infty$, and $E(1)=1$.

The following identity is an easy consequence of Theorem \ref{th:riesz} (see
Section~\ref{se:co:riesz} and Remark \ref{rm:invlap}).

\begin{corollary}[Special function identity\footnote{It is worth noting that
    $4r^2\lambda^2\leq(\lambda^2+r^2)^2$ with equality if and only if
    $\lambda=r$, so that the radius of convergence of the ${}_2F_1$ in
    \eqref{eq:co:riesz}, which is equal to $1$, is reached in the interior of
    the interval of integration over $r$ when $\lambda\in[0,1)$.%
  }]\label{co:riesz} For $d\in\{2,3,4,\ldots\}$, $d-2<s<d$,
$\lambda\in[0,1]$,
  \begin{equation}\label{eq:co:riesz}
    \int_0^1
    {}_2F_1\left(\frac{s}{4},\frac{s+2}{4};\frac{d}{2};\frac{4 r^2 \lambda ^2}{(\lambda^2+r^2)^2}\right)
    \frac{r^{d-1}}{(\lambda
      ^2+r^2)^{\frac{s}{2}}(1-r^2)^{\frac{d-s}{2}}}\,\mathrm{d}r
    =\frac{\pi}{2\sin(\frac{\pi}{2}(d-s))}.
  \end{equation}
\end{corollary}

Here are some special cases worth noting for the hypergeometric function in
\eqref{eq:co:riesz}:
\begin{itemize}
\item for $(d,s)=(2,1)$ we get
  ${}_2F_1\left(\frac{s}{4},\frac{s+2}{4};\frac{d}{2};z\right)=
  {}_2F_1\left(\frac{1}{4},\frac{3}{4};1;z\right)=
  \frac{2 K\left(\frac{2 \sqrt{z}}{\sqrt{z}+1}\right)}{\pi \sqrt{\sqrt{z}+1}}$.\\[1ex]
\item for $(d,s)=(3,2)$ we get
  ${}_2F_1\left(\frac{s}{4},\frac{s+2}{4};\frac{d}{2};z\right)=
  {}_2F_1\left(\frac{1}{2},1;\frac{3}{2};z\right)=
  \frac{\tanh^{-1}(\sqrt{z})}{\sqrt{z}}$. \\[1ex]
\item for $(d,s)=(5,4)$ we get
  ${}_2F_1\left(\frac{s}{4},\frac{s+2}{4};\frac{d}{2};z\right)=
  {}_2F_1\left(1,\frac{3}{2};\frac{5}{2};z\right)=
  \frac{3 \left(\sqrt{z} \tanh^{-1}\left(\sqrt{z}\right)-z\right)}{z^2}$.\\[1ex]
\end{itemize}

Our main potential theoretic result is the following external field version of Theorem
\ref{th:riesz}.

\begin{theorem}[Main result]\label{th:d=s+3}
  Suppose that $d\in\{2,3,4,\ldots\}$ and $s=d-3$, namely
  \[
    (d,s)\in\{(2,-1),(3,0),(4,1),\ldots\}.
  \]
  Let
  \begin{equation}\label{eq:d=s+3RV}
    R:=
    \left(\frac{c_{d,d-3}\sqrt{\pi}}{4\gamma}
      \frac{\Gamma\bigr(\frac{d+1}{2}\bigr)}{\Gamma\bigr(\frac{d+2}{2}\bigr)}\right)^{\frac{1}{d-1}},
    \quad\text{where}\quad
    c_{d, s} := \begin{cases}
      |s| (d-2-s) & \text{if  $s\neq 0$} \\
      d-2 & \text{if $s = 0$}
    \end{cases}.
  \end{equation}
  If $V = \gamma\left|\cdot\right|^2$, $\gamma>0$, then the equilibrium
  measure $\mu_{\mathrm{eq}}$ for the minimum energy problem on $\mathbb{R}^d$
  (\ref{eq:EsV}--\ref{eq:equ}) with kernel $K_s$ and external field $V$ is the
  ``radial arcsine distribution''
  \begin{equation}\label{eq:d=s+3meq}
    \mu_{\mathrm{eq}}(\mathrm{d} x)
    =
    \frac{\Gamma\bigr(\frac{d+1}{2}\bigr)}{\pi^{\frac{d+1}{2}}R^{d-1}}
    \frac{\mathbf{1}_{|x|\leq R}}{\sqrt{R^2-|x|^2}}
    \mathrm{d} x
    =
    \frac{2\Gamma\bigr(\frac{d+1}{2}\bigr)}{\sqrt{\pi}\Gamma(\frac{d}{2})R^{d-1}}
    \frac{r^{d-1}\mathbf{1}_{r\leq R}}{\sqrt{R^2-r^2}}
    \,\mathrm{d}r\,\mathrm{d} \sigma_{S_1},
  \end{equation}
  where $\mathrm{d} x$ and $\mathrm{d} r$ are the Lebesgue measures on $\mathbb{R}^d$ and on
  $[0,\infty)$ respectively. Moreover, this $\mu_{\mathrm{eq}}$ is also the equilibrium
  measure in Theorem \ref{th:riesz} with $s=d-1$ and $R$ as in
  \eqref{eq:d=s+3RV}.
\end{theorem}

Theorem~\ref{th:d=s+3} is proved in Section \ref{se:th:d=s+3}.

Several extensions of Theorem \ref{th:d=s+3} for more general $(d,s)$ or $V$
are considered in \cite{hdep}.

\begin{table}[htbp]
  \centering
  \begin{tabular}{c|c|c|c}
    $d=2,s=-1$
    & $d=3,s=0$
    & $d=4,s=1$
    & $d=\infty,s=\infty-3$\\\hline
    $\displaystyle\frac{\pi}{8}\approx 0.392699$
    & $\displaystyle\frac{1}{\sqrt{3}}\approx 0.57735$
    & $\displaystyle\frac{1}{2}\sqrt[3]{\frac{3\pi}{4}}\approx 0.665335$
    & $1$\\[1em]
  \end{tabular}
  \caption{\label{tb:d=s+3} Some special values of the critical radius $R$ in
    Theorem \ref{th:d=s+3}, with $\gamma=1$.}
\end{table}
Table~\ref{tb:d=s+3} gives values of the radius $R$ in \eqref{eq:d=s+3RV} for
$\gamma = 1$ and various values of $d$. For $\gamma = 1$, integer $d \geq 2$,
the function $d\mapsto R$ achieves its minimum $\approx0.392699$ at $d=2$
($s=-1$) and its maximum $\approx 1.04747$ at $d=16$ ($s=13$), and these
values are the unique extreme points. Regarding high dimensional behavior or
asymptotic analysis, we have $\lim_{d=s+3\to\infty}R=1$.

As we shall verify, Theorem \ref{th:d=s+3} yields the following integral
formulas.
\begin{corollary}[Integral formula]\label{co:formulas}
  Let $d\in\{2,3,4,\ldots\}$ and $\lambda\in[0,1]$. Then
  \begin{equation}\label{eq:co:formula:sin}
    \int_0^1\mathcal{S}_{d-3}\Bigr(\frac{4\lambda r}{(\lambda+r)^2}\Bigr)
    \frac{(\lambda+r)^{3-d}r^{d-1}\,\mathrm{d}r}{\sqrt{1-r^2}}
    =\frac{\pi^{\frac{3}{2}}\Gamma(\frac{d-1}{2})}{2^{d+1}\Gamma(\frac{d}{2})}\Bigr(\Bigr(\frac{3}{d}-1\Bigr)\lambda^2
    +1\Bigr),
  \end{equation}
  where
  \begin{equation}
    \mathcal{S}_s(z):=\int_0^{\frac{\pi}{2}}\!\frac{\sin^{s+1}(2\alpha)\mathrm{d}\alpha}{2^{s+1}(1-z\sin^2(\alpha))^{\frac{s}{2}}}
    =\int_0^1\!\frac{t^{s+1}(1-t^2)^{\frac{s}{2}}\mathrm{d} t}{(1-zt^2)^{\frac{s}{2}}}
    =\frac{\Gamma(\frac{s+2}{2})^2}{2\Gamma(s+2)}\,{}_2F_1\Bigr(\frac{s+2}{2},\frac{s}{2};s+2;z\Bigr)\label{eq:co:formula:sinS}.
  \end{equation}
  Equivalently, for $d\in\{2,3,4,\ldots\}$ and $\lambda\in[0,1]$,
    \begin{equation}\label{eq:co:formula:sinF}
      \int_0^1{}_2F_1\Bigr(\frac{d-1}{2},\frac{d-3}{2};d-1;\frac{4\lambda r}{(\lambda+r)^2}\Bigr)\frac{(\lambda+r)^{3-d}r^{d-1}}{\sqrt{1-r^2}}\,\mathrm{d}r
      =\frac{\pi}{4}\Bigr(\Bigr(\frac{3}{d}-1\Bigr)\lambda^2
      +1\Bigr).
  \end{equation}
\end{corollary}

Corollary \ref{co:formulas} is proved in Section \ref{se:co:formulas}.

The formula \eqref{eq:co:formula:sinF} comes from the Euler\,--\,Lagrange
characterization related to Theorem \ref{th:d=s+3}. Numerical experiments
suggest that \eqref{eq:co:formula:sin} and \eqref{eq:co:formula:sinF} remain
valid whenever the parameter $d>1$ is real.

It is tempting to regard $\mathcal{S}_s$ as a \emph{special function} in its own right.
When $(d,s)=(2,-1)$, it becomes the complete elliptic integral of the second
kind, namely $\mathcal{S}_{-1} = E$, and \eqref{eq:co:formula:sin} becomes
\eqref{eq:co:formulas:E} below.

\begin{corollary}[More integral formulas\footnote{Note that
    $\frac{4r\lambda}{(\lambda+r)^2}=\frac{4x}{(1+x)^2}$ with
    $x:=\frac{\lambda}{r}$. The map $x \mapsto \frac{4x}{(1+x)^2}$ is the
    \emph{Landen transform} of elliptic integrals.}]\label{co:formulas:more}
  For $\lambda\in[0,1]$,
  \begin{align}
    \int_0^1
    \bigr((\lambda+r)^2\log(\lambda+r)-(\lambda-r)^2\log|\lambda-r|\bigr)
    \frac{r\,\mathrm{d}r}{\sqrt{1-r^2}}
    &=\pi\Bigr(\frac{\lambda^3}{3}+(1-\log 2)\lambda\Bigr),\label{eq:co:formulas:log}\\
    \int_0^1((\lambda +r)\log(\lambda +r)-(\lambda-r)\log|\lambda-r|)\frac{r\,\mathrm{d}r}{\sqrt{1-r^2}}
    &=\frac{\pi}{2}\Bigr(\lambda^2+\frac{1}{2}-\log 2\Bigr),\label{eq:co:formulas:log'}\\
    \int_0^1E\left(\frac{4\lambda r}{(\lambda+r)^2}\right)
    \frac{(\lambda+r)r\,\mathrm{d}r}{\sqrt{1-r^2}}
    &=\frac{\pi^2}{8}\Bigr(\frac{\lambda^2}{2}+1\Bigr),\label{eq:co:formulas:E}\\
    \int_0^1
      K\left(\frac{4\lambda r}{(r+\lambda)^2}\right)
    \frac{(\lambda-r)r\,\mathrm{d}r}{\sqrt{1-r^2}}
    &=\frac{\pi^2}{8}\Bigr(\frac{3\lambda^2}{2}-1\Bigr),
      \label{eq:co:formulas:K}
  \end{align}
  where $E$ and $K$ are the special functions defined in \eqref{eq:EllipticE}
  and \eqref{eq:EllipticK}.
\end{corollary}

Corollary \ref{co:formulas:more} is proved in Section
\ref{se:co:formulas:more} by applying further transformations to the
Euler\,--\,Lagrange conditions of Theorem \ref{th:d=s+3} in the special cases
$(d,s)\in\{(2,-1),(3,0)\}$.

To the best of our knowledge, the formulas provided by Corollaries
\ref{co:riesz}, \ref{co:formulas}, and \ref{co:formulas:more} are not found in
the existing catalogs of identities and tables for series and integrals such
as \cite{MR0058756}, \cite{MR1054647}, and \cite{MR2460394}, and are not recognized by advanced
software such as Maplesoft Maple and Wolfram Mathematica. However it is worth
noting that these softwares do recognize the first parts of
\eqref{eq:co:formulas:log} and \eqref{eq:co:formulas:log'} in terms of
${}_3F_2$:
\begin{multline}
  \int_0^1(\lambda +r)^2\log(\lambda +r)\frac{r \,\mathrm{d}r}{\sqrt{1-r^2}}
  \\=
  \frac{\pi  \, _3F_2\left(\frac{1}{2},1,1;2,3;\frac{1}{\lambda ^2}\right)}{16 \lambda }-\frac{2 \, _3F_2\left(1,1,\frac{3}{2};\frac{5}{2},\frac{7}{2};\frac{1}{\lambda ^2}\right)}{45 \lambda ^2}+\frac{\pi  \lambda }{4}+\frac{3 \lambda  (2 \lambda +\pi )+4}{6}\log (\lambda )+1\label{eq:co:formulas:log2}
\end{multline}
and
\begin{multline}
  \int_0^1(\lambda +r)\log(\lambda +r)\frac{r \,\mathrm{d}r}{\sqrt{1-r^2}}\\
  =
  \frac{32 \lambda  \, _3F_2\left(\frac{1}{2},1,1;\frac{3}{2},\frac{5}{2};\frac{1}{\lambda ^2}\right)-3 \pi  \, _3F_2\left(1,1,\frac{3}{2};2,3;\frac{1}{\lambda ^2}\right)+24 \lambda ^2 ((4 \lambda +\pi ) \log (\lambda )+\pi )}{96 \lambda ^2}.\label{eq:co:formulas:log'2}
\end{multline}

\section{Proofs}

\subsection{Proof of Corollary \ref{co:riesz}}
\label{se:co:riesz}

Let us consider the settings of Theorem \ref{th:riesz} in the case $d-2<s<d$.
By scaling we can assume without loss of generality that $R=1$. Then, using
the Funk\,--\,Hecke formula \eqref{eq:funkhecke}, we get, for
$x\in\mathbb{R}^d$, $x\neq0$, denoting $\lambda:=|x|$ and
$C:=\frac{\Gamma(1+\frac{s}{2})}{\pi^{\frac{d}{2}}\Gamma(1+\frac{s-d}{2})}$,
\begin{align*}
  U^{\mu_{\mathrm{eq}}}(x)
  &=C\int_{|y|\leq 1}\frac{|x-y|^{-s}}{(1-|y|^2)^{\frac{d-s}{2}}}\mathrm{d} y\\
  &=C\int_0^1\Bigr(\int_{S_1}\frac{(\lambda^2+r^2-2\lambda
    r\frac{x}{|x|}\cdot u)^{-\frac{s}{2}}}{(1-r^2)^{\frac{d-s}{2}}}r^{d-1}\mathrm{d} u\Bigr)\,\mathrm{d}r\\
  &=C\tau_{d-1}|S_1|
    \int_0^1\Bigr(\int_{-1}^1\frac{(1-t^2)^{\frac{d-3}{2}}}{(\lambda^2+r^2-2\lambda rt)^{\frac{s}{2}}}\mathrm{d} t\Bigr)\frac{r^{d-1}}{(1-r^2)^{\frac{d-s}{2}}}\,\mathrm{d}r.
\end{align*}
Now, since $\frac{2\lambda r}{(\lambda^2+r^2)}\in[0,1/2]$, using the Newton
binomial series \eqref{eq:binom},
\begin{align*}
  \int_{-1}^1\frac{(1-t^2)^{\frac{d-3}{2}}}{(\lambda^2+r^2-2\lambda rt)^{\frac{s}{2}}}\mathrm{d} t
  &=\frac{1}{(\lambda^2+r^2)^{\frac{s}{2}}}\int_{-1}^1\frac{(1-t^2)^{\frac{d-3}{2}}}{(1-\frac{2\lambda
    r}{\lambda^2+r^2}t)^{\frac{s}{2}}}\mathrm{d} t\\
  &=\frac{1}{(\lambda^2+r^2)^{\frac{s}{2}}}\sum_{n=0}^\infty
    \frac{(\frac{s}{2})_{2n}}{(2n)!}\Bigr(\frac{4r^2\lambda^2}{(\lambda^2+r^2)^2}\Bigr)^n\int_{-1}^1(1-t^2)^{\frac{d-3}{2}}t^{2n}\mathrm{d} t\\
  &=\frac{1}{(\lambda^2+r^2)^{\frac{s}{2}}}\sum_{n=0}^\infty
    \frac{(\frac{s}{2})_{2n}}{(2n)!}\Bigr(\frac{4r^2\lambda^2}{(\lambda^2+r^2)^2}\Bigr)^n
  \frac{\Gamma(\frac{d-1}{2})\Gamma(n+\frac{1}{2})}{\Gamma(\frac{d}{2}+n)}\\
  &=\frac{\sqrt{\pi }\Gamma(\frac{d-1}{2})}{\Gamma(\frac{d}{2})}\frac{\,
    _2F_1\left(\frac{s}{4},\frac{s+2}{4};\frac{d}{2};\frac{4 r^2 \lambda
    ^2}{\left(r^2+\lambda ^2\right)^2}\right)}{(\lambda
    ^2+r^2)^{\frac{s}{2}}},
\end{align*}
where we have use the identities (related to the Legendre duplication formula
\eqref{eq:duplication})
\[
  \Bigr(\frac{s}{2}\Bigr)_{2n}
  =2^{2n}\Bigr(\frac{s}{4}\Bigr)_n
  \Bigr(\frac{s}{4}+\frac{1}{2}\Bigr)_n
  \quad\text{and}\quad
  2^{2n}\Bigr(\frac{1}{2}\Bigr)_n
  =\frac{(2n)!}{n!}.
\]
Note that we could alternatively proceed as in the proof of Lemma
\ref{le:landend} and use the Euler integral formula \eqref{eq:2F1Int}. Next,
using the fact that
$\tau_{d-1}=\frac{\Gamma(\frac{d}{2})}{\sqrt{\pi}\Gamma(\frac{d-1}{2})}$, we
get, for $x\in\mathbb{R}^d$, $x\neq0$,
\begin{align*}
  U^{\mu_{\mathrm{eq}}}(x)
  &=
    C|S_1|
    \int_0^1
    \, _2F_1\left(\frac{s}{4},\frac{s+2}{4};\frac{d}{2};\frac{4 r^2 \lambda ^2}{(\lambda^2+r^2)^2}\right)
    \frac{r^{d-1}}{(\lambda
    ^2+r^2)^{\frac{s}{2}}(1-r^2)^{\frac{d-s}{2}}}\,\mathrm{d}r.
\end{align*}
(This last integral may be compared with extensions of the Beta integral in
\cite[Th.~5.1]{MR2281163} and \cite{MR2090496}). %

Now, the Euler\,--\,Lagrange conditions \eqref{eq:eulerlagrange} and the
continuity of $U^{\mu_{\mathrm{eq}}}$ give that this quantity is constant on $B_1$. Finally
the value of the constant can be obtained from \eqref{eq:riesz}.
\hfill\qedsymbol

\subsection{Proof of Theorem \ref{th:d=s+3}}
\label{se:th:d=s+3}

We split the proof into several subsections.

\subsubsection{Computation of critical radius and candidate
  equilibrium measure}

From the uniqueness property we know that the equilibrium measure $\mu_{\mathrm{eq}}$ is
radially symmetric. Let us make a succession of assumptions to extract a
candidate for $\mu_{\mathrm{eq}}$, and we will then check that it indeed satisfies the
Euler\,--\,Lagrange conditions \eqref{eq:eulerlagrange}. We start by observing
from \eqref{eq:eulerlagrange} that, for $x\in S_*:=\mathrm{supp}(\mu_{\mathrm{eq}})$,
\begin{equation}\label{eq:ELK}
  \int K_s(x-y)\mu_{\mathrm{eq}}(\mathrm{d}  y)+\gamma|x|^2=c.
\end{equation}
Applying the Laplacian operator to \eqref{eq:ELK} and assuming it can be taken
inside the integral, we get from \eqref{eq:ELK} and \eqref{eq:DeltaK} that for
all $x$ in the interior of $S_*$,
\begin{equation}\label{eq:mustar3deuler}
  -c_{d,s}\int K_{s+2}(x-y)\mu_{\mathrm{eq}}(\mathrm{d}  y)+2\gamma d=0.
\end{equation}
In our case $s=d-3$, so $c_{d,s} = c_{d,d-3}$ is equal to $|d-3|$ if $d\neq3$
while it is equal to $1$ if $d=3$.

Next suppose that $S_* = B_R$ for some $R>0$. Let $\nu_R$ be the equilibrium
measure for the minimum energy problem on $B_R$ with kernel
$K_{s+2}=\left|\cdot\right|^{-(s+2)}$ and $V = 0$. Observing that $\nu_R$ is
the dilation by a factor of $R$ of $\nu_1$, we see from Theorem \ref{th:riesz}
that $\nu_R$ is the ``radial arcsine distribution''; in other words, the
measure
\begin{equation}\label{eq:nur}
  \nu_R(\mathrm{d} x)  =
  \frac{C_{d,R}}{\sqrt{R^2-|x|^2}}\mathbf{1}_{|x|\leq R} \; \mathrm{d} x,
  \quad\text{where}\quad
  C_{d,R}=\frac{2\Gamma\bigr(\frac{d+1}{2}\bigr)}{|S_1|\sqrt{\pi}\Gamma\bigr(\frac{d}{2}\bigr)R^{d-1}}
  =\frac{\Gamma\bigr(\frac{d+1}{2}\bigr)}{R^{d-1}\pi^{\frac{d+1}{2}}}.
\end{equation}
In particular, the support of $\nu_R$ is all of $B_R$. Next, by definition of
$\nu_R$, the associated Euler\,--\,Lagrange conditions state that, for some
constant $W_R$,
\begin{equation}\label{eq:nurwr}
  \int K_{s+2}(x-y)\nu_R(\mathrm{d}  x) = W_R,
  \quad \text{for } y\in B_R.
\end{equation}
As $d = s+3$ we obtain (using $y = 0 \in B_R$)
\begin{equation}\label{eq:WR}
  W_R = \frac{W_1}{R^{s+2}} = \frac{W_1}{R^{d-1}}
  \quad\text{and}\quad
  W_1=C_{d,1}|S_1|\int_0^1\frac{r^{-(s+2)}r^{d-1}}{\sqrt{1-r^2}}\,\mathrm{d}r
  =\sqrt{\pi}\frac{\Gamma\bigr(\frac{d+1}{2}\bigr)}{\Gamma\bigr(\frac{d}{2}\bigr)}.
\end{equation}
To derive the value of $R$ we first integrate \eqref{eq:mustar3deuler} with respect to $\nu_R(\mathrm{d} x)$ and
swap the integrals, assuming that this is legal, giving
\begin{equation}\label{eq:nuralpha}
  -c_{d,s}\int\Bigr(\int K_{s+2}(x-y)\nu_R(\mathrm{d}  x)\Bigr)\mu_{\mathrm{eq}}(\mathrm{d}  y)
  +2\gamma d=0.
\end{equation}
Then, using \eqref{eq:nurwr} and \eqref{eq:WR} in \eqref{eq:nuralpha}, we get
\[
  \frac{c_{d,s}\sqrt{\pi}\Gamma\bigr(\frac{d+1}{2}\bigr)}{\Gamma\bigr(\frac{d}{2}\bigr)}
  R^{-s-2}
  =2\gamma d.
\]
Finally, from the formula $z\Gamma(z)=\Gamma(z+1)$ with $z=d/2$,
we derive the desired formula for $R$, namely
\begin{equation}\label{eq:R}
  R = \left( \frac{c_{d,d-3}\sqrt{\pi} \Gamma(\frac{d+1}{2})} {4\gamma\Gamma(\frac{d+2}{2})}\right)^{\frac{1}{d-1}} =
  \left( \frac{c_{s+3,s}\sqrt{\pi} \Gamma(\frac{s+4}{2})} {4\gamma\Gamma(\frac{s+5}{2})}\right)^{\frac{1}{s+2}}.
\end{equation}
See also Remark \ref{rm:invlap} for an alternative way to compute this
critical value of $R$.

\subsubsection{Euler\,--\,Lagrange characterization}

The probability measure $\nu_R$ in \eqref{eq:nur} with $R$ as in \eqref{eq:R}
satisfies the Frostman conditions \eqref{eq:eulerlagrange} with kernel $K_s$
and $V=\gamma\left|\cdot\right|^2$ thanks to Lemma \ref{le:Phi} below.

\begin{lemma}[Potentials]\label{le:Phi}
  Let $R$ be as in \eqref{eq:R} and define
  $\Phi := K_s*\nu_R+\gamma\left|\cdot\right|^2$. Then,
  \begin{itemize}
  \item $\Phi$ is continuous on $\mathbb{R}^d$;
  \item $\Phi=\Phi(0)$ on $B_R$;
  \item $\Phi\geq\Phi(0)$ outside $B_R$.
  \end{itemize}
\end{lemma}

\begin{proof}
  As 
  $K_s*\nu_R$ is radially symmetric, so is $\Phi$. Thus we define for any
  $\lambda \geq 0$,
  \begin{equation} \label{Eq:Phidef}
    \varphi(\lambda) := \Phi\left(\lambda R \widehat{x}\right)
     \quad\mbox{ for any } \widehat{x}\in\mathbb{R}^d \mbox{ with } |\widehat{x}| = 1.
  \end{equation}
  Using from Lemma \ref{le:rieszpot} that $K_s*\nu_R\in L^1_{\mathrm{loc}}(\mathbb{R}^d,\mathrm{d}x)$ and we can swap the Laplacian and the Riesz potential, we get
  \begin{equation}\label{eq:DeltaPhi}
    \Delta\Phi
    =-c_{s+3,s}K_{s+2}*\nu_R+2d\gamma.
  \end{equation}
  Moreover, $\nu_R$ and the radius $R$ have been chosen in the preceding
  subsection precisely in such a way that on
  $\mathrm{int}(B_R):=\{x\in\mathbb{R}^d:|x|<R\}$,
  \begin{equation}\label{eq:DeltaPhi0}
    \Delta\Phi
    =-c_{s+3,s}K_{s+2}*\nu_R+2d\gamma
    =0.
  \end{equation}

  \emph{Continuity of $\Phi$.} At this step, let us remark that when $d<6$,
  Lemma \ref{le:rieszpot} \emph{(iv)} gives that $K_s*\nu_R$ is continuous on
  $\mathbb{R}^d$ since $\frac{\mathrm{d}\nu_R}{\mathrm{d} x}\in L^p(\mathbb{R}^d,\mathrm{d} x)$,
  $2>p>d/(d-s)=d/3$ (recall that $s=d-3$).

  Actually $K_s*\nu_R$ is continuous on $\mathbb{R}^d$ for arbitrary dimension
  $d$. Indeed, this can be checked directly, in the case $(d,s)=(2,-1)$, since
  $K_{-1}$ is not singular. This can also be checked directly in the case
  $(d,s)=(3,0)$ from the formula provided by Lemma \ref{le:varphi}. Finally,
  in the case $s=d-3>0$, using Lemma \ref{le:varphi} and Lemma
  \ref{le:landend} and the change of variable $r=\sin(\theta)$ to remove the
  singularity at the edge $r=1$, we get, for $x\in\mathbb{R}^d$, with
  $\lambda=|x|/R$, %
  for some constant $C_s>0$,
  \begin{equation}\label{eq:Ksnucont}
    (K_s*\nu_R)(x)=C_s\int_0^{\frac{\pi}{2}}
    {}_2F_1\Bigr(
    \frac{s}{2}+1,
    \frac{s}{2};
    s+2;
    \frac{4\lambda\sin(\theta)}{(\lambda+\sin(\theta))^2}
    \Bigr)
    \frac{\sin(\theta)^{s+2}}{(\lambda+\sin(\theta))^s}\mathrm{d}\theta.
  \end{equation}
  The continuity of $K_s*\nu_R$ follows then from the uniform continuity of
  the hypergeometric function. Indeed, by \eqref{eq:2F1z1} the series that
  defines ${}_2F_1(a,b;c;z)$ converges absolutely for $z\in[0,1]$ (and
  remarkably for $z=1$) when $c-a-b>0$ as in our case $c-a-b=1$. The
  hypergeometric function
  ${}_2F_1\left(\frac{s}{2}+1,\frac{s}{2};s+2;z\right)$ is uniformly
  continuous on $[0, 1]$ since it is clearly analytic on $[0, 1)$ and it is
  also continuous at $z = 1$; the latter assertion follows from Abel's Limit
  Theorem~\cite[Sec.~2.5]{MR510197} and the fact that
  ${}_2F_1\left(\frac{s}{2}+1,\frac{s}{2};s+2;1\right)$ is finite. Furthermore
  $\lambda=0$ is not a problem as soon as we establish the fact that $\Phi$ is
  harmonic (in fact constant) in the unit disk (see below!).

  \emph{Constantness on $B_R$.} It follows from Lemma~\ref{le:liouville}.

  \emph{Behavior outside $B_R$.} Since $\varphi$ defined in \eqref{Eq:Phidef}
  is continuous on $[0,+\infty)$ and differentiable on $(1,+\infty)$, the
  Frostman condition outside $B_R$ is realized if we show that
  $\varphi'(\lambda)\geq0$ for $\lambda>1$.

  Let us first consider the case $(d,s)=(2,-1)$. By Lemma
  \ref{le:varphi:d2s-1}, $\varphi$ is convex on $[1,+\infty)$ since
    \begin{equation}\label{eq:d=2s=-1phipp}
      \varphi''(\lambda)
      =
      R\frac{\sqrt{1-\frac{1}{\lambda ^2}}}{2 \lambda
      }-R\frac{\arcsin\left(\frac{1}{\lambda }\right)}{2}+2\gamma R^2
      \geq 0-\frac{\pi}{8\gamma}\frac{\frac{\pi}{2}}{2}+2\frac{\pi^2}{64\gamma}
      =0,\quad \lambda>1.
    \end{equation}
    Moreover since \eqref{eq:d=2s=-1phip3} gives
    $\lim_{\lambda\to1^+}\varphi'(\lambda)=-\frac{\pi}{16\gamma}\frac{\pi}{2}+2\frac{\pi^2}{64\gamma}=0$,
    it follows that $\varphi'(\lambda)\geq0$ when $\lambda>1$.

    Finally, consider the case $(d,s)=(3,0)$. By Lemma
    \ref{le:varphi:d3s0}, if $\lambda>1$,
    \begin{equation}
      \varphi'(\lambda)
      =\frac{2\sqrt{\lambda^2-1}^3}{3\gamma\lambda^2}\geq0.
    \end{equation}

    Let us now consider the case $s=d-3>0$. We can rewrite \eqref{eq:Ksnucont}
    as
    \begin{equation}
      \varphi(\lambda)
      =C_s\int_0^1 h(\lambda, r) \; \frac{r^{d-1}}{\sqrt{1 - r^2}} \;\mathrm{d} r ,
      \label{eq:pot1}
    \end{equation}
    where
    \begin{align}
      h(\lambda, r)
      & :=
        {}_2F_1\left(\frac{s}{2}+1,\frac{s}{2};s+2;
        \frac{4\lambda r}{(\lambda + r)^2}\right)
        \left(\lambda + r\right)^{-s} \label{eq:h1}\\
      & =
        {}_2F_1\left(\frac{s}{4}, \frac{s+2}{4}; \frac{s+3}{2};
        \frac{4\lambda^2 r^2}{(\lambda^2 + r^2)^2}\right)
        \left(\lambda^2 + r^2\right)^{-\frac{s}{2}} \label{eq:h2}
    \end{align}
    where the last equality comes from the quadratic transformation
    \eqref{eq:2F1Q}. Differentiating \eqref{eq:h1} with
    \[
      z = \frac{4\lambda r}{(\lambda + r)^2}, \qquad
      \frac{\partial z}{\partial \lambda} = \frac{4 r (r - \lambda)}{(\lambda + r)^3}
    \]
    and using the derivative formula \eqref{eq:2F1deriv} for ${}_2F_1$ we get
    \begin{multline}
      \frac{\partial h}{\partial \lambda}(\lambda, r)
      =  s r (r-\lambda) (\lambda + r)^{-s-3}
      \ {}_2F_1\left(\frac{s}{2}+2,\frac{s}{2}+1; s+3;
        \frac{4\lambda r}{(\lambda + r)^2}\right) \\
      - s (\lambda + r)^{-s-1}
      \ {}_2F_1\left(\frac{s}{2}+1,\frac{s}{2}; s+2;
        \frac{4\lambda r}{(\lambda + r)^2}\right) . \label{eq:Dh}
    \end{multline}
    The only potential difficulties are when the argument of the
    hypergeometric functions is $z = 1$. As before $z\in [0, 1]$ and
    $z = 1 \Longleftrightarrow \lambda = r$. The parameters in the first
    hypergeometric function in \eqref{eq:Dh} satisfy $c - a - b = 0$, so by
    the property \eqref{eq:2F1z1a} of ${}_2F_1$, we get
    \[
      \lim_{\lambda\to r} (r - \lambda) \  {}_2F_1\left(\frac{s}{2}+2,\frac{s}{2}+1;s+3;
        \frac{4\lambda r}{(\lambda + r)^2}\right) = 0.
    \]
    As before, the second hypergeometric function in \eqref{eq:Dh} has
    parameters which satisfy $c - a - b = 1 >0$. Thus
    $\frac{\partial h}{\partial \lambda}(\lambda, r)$ is uniformly continuous
    for $r\in [0, 1]$ and $\lambda \geq 0$, so by the Leibniz integral rule
    \[
      \varphi'(\lambda) =
      C_s\int_0^{\frac{\pi}{2}} \frac{\partial h}{\partial \lambda}(\lambda, \sin(\theta)) \;
      \sin(\theta)^{d-1} \; \mathrm{d} \theta
    \]
    and $\varphi'(\lambda)$ is continuous for $\lambda \geq 0$. In particular
    \[
      \lim_{\lambda\to 1^+} \varphi'(\lambda) = \varphi'(1).
    \]
    Let us show now that $\varphi'(\lambda)\geq0$ for $\lambda\geq0$.

    Since $s+2=d-1>d-2$, the function $K_{s+2}*\nu_R$ is subharmonic outside
    the support $B_R$ of $\nu_R$, see for instance
    \cite[Th.~I.1.4~p.~66]{MR0350027}. Since it is continuous everywhere in
    $\mathbb{R}^d$ even at $\infty$, it follows by the maximum principle
    applied on the complement of $B_R$, that for $|x| \geq R$,
    \begin{equation}\label{eq:RieszSubH}
      I(\nu_R) \geq U^{\nu_R}(x) =
      \int_{\mathbb{R}^d} \frac{1}{|x - y|^{s+2}}\; \nu_R(\mathrm{d} y)
    \end{equation}
    (equality holds for $|x| = R$ by Theorem \ref{th:riesz}). It follows by
    using \eqref{eq:rieszumeq} and \eqref{eq:DeltaPhi} that
    $\Delta \Phi(x) \geq 0$ for $|x|>R$. Next, using the radial form of
    the Laplacian,
    \[
      \frac{1}{\lambda^{d-1}}\left(\lambda^{d-1} \varphi'(\lambda) \right)' \geq 0
      \quad\text{for}\quad \lambda >1.
    \]
    Thus
    \(\displaystyle \int_\rho^\lambda [\tau^{d-1} \varphi'(\tau) ]' \; \mathrm{d}
    \tau \geq 0 \text{ for }\lambda \geq \rho> 1, \) and so
    \[
      \lambda^{d-1} \varphi'(\lambda) \geq \rho^{d-1} \varphi'(\rho).
    \]
    Finally, letting $\rho \to 1^+$ we get, as $\varphi'(1) = 0$,
    \[
      \varphi'(\lambda) \geq 0 \quad\text{for}\quad \lambda \geq 1.
    \]
    We also know that $\varphi$ is constant for $0 \leq \lambda \leq 1$, so
    $\varphi'(\lambda) \geq 0$ for $\lambda \geq 0$.
\end{proof}

\begin{lemma}[Laplacian inversion or Liouville lemma]\label{le:liouville}
  Let $\Phi:\mathrm{int}(B_R)\to\mathbb{R}$, $d\geq2$, $R>0$.
  If
  \begin{itemize}
  \item (local integrability) $\Phi\in L^1_{\mathrm{loc}}(\mathrm{d}x)$;
  \item (weak harmonicity) $\Delta\Phi=c$ for a constant $c$, in the sense of
    Schwartz distributions;
  \item (radial symmetry) $\Phi$ is equal to a constant on
    the sphere $S_r$ for all $r<R$;
  \end{itemize}
  then $\Phi$ is $\mathcal{C}^\infty$ and is given by
  $\Phi=\frac{c}{2d}\left|\cdot\right|^2+\Phi(0)$.
  In particular $\Phi$ is constant when $c=0$.
\end{lemma}

Related statements can be found in \cite[Sec.~0.3]{MR1485778} for $d=2$, and
in \cite[Th.~3.3, Ch.~III, p.~183]{MR0350027}.

\begin{remark}[Extension] Lemma \ref{le:liouville} extends to the case where
  $\Delta\Phi$ is a $\mathcal{C}^\infty$ radial function on $B_R$, say
  $\Delta\Phi=A(\left|\cdot\right|)$. Indeed the same proof gives
  $\Phi=\Phi(0)+B(\left|\cdot\right|)$, where $B$ solves
  $rB''(r)+(d-1)B'(r)=A(r)$, $0<r<R$, with $B(0)=B'(0)=0$, which gives
  $B(r)=\int_0^ru^{1-d}\Bigr(\int_0^u v^{d-1}A(v)\mathrm{d} v\Bigr)\mathrm{d} u$. Thus
  \[
    B(r)=\int_0^r(k(r)-k(v))v^{d-1}A(v)\mathrm{d} v
    \quad
    \text{with}
    \quad
    k(v):=\begin{cases}\frac{v^{2-d}}{2-d}&\text{if
        $d\neq2$}\\\log(v)&\text{if $d=2$}\end{cases}.
  \]
  If $A$ is a polynomial of degree $m$, then $B$ is a polynomial of degree
  $m+2$, while if $A$ is a hypergeometric series, then $B$ is also a
  hypergeometric series. For an arbitrary integer $m\geq1$, repeating this
  procedure gives a symmetric polynomial in $d$ variables $\Phi$ such that
  $\Delta^m\Phi=c$.
  See for instance 
  \cite{GutlebCarrilloOlver2021Balls}
  and references therein for a link with Jacobi and Zernike orthogonal
  polynomials and hypergeometric functions.
\end{remark}

\begin{proof}[Proof of Lemma \ref{le:liouville}]
  By a version of the Weyl lemma expressing the Hörmander hypoellipticity of
  the Laplacian operator, see for instance Stroock's expository note
  \cite{MR2528466}, we get that $\Phi$ is
  $\mathcal{C}^\infty(\mathrm{int}(B_R))$. Next, by radial symmetry
  $\Phi(x)=\psi(r)$ where $r=|x|$. Using
  $\Delta=\partial_r^2+\frac{d-1}{r}\partial_r+\frac{1}{r^2}\Delta_{S_1}$ we
  get that
  \[ 
  c=\Delta\Phi(x)=\psi''(r)+\frac{d-1}{r}\psi'(r)=\frac{(r^{d-1}\psi'(r))'}{r^{d-1}},
  \]
  and thus $r^{d-1}\psi'(r)=\frac{c}{d}r^d$ (note that we use here the fact
  that $d>1$ to get that $r^{d-1}\psi'(r)=0$ when $r\to0$). Hence
  $\psi(r)=\frac{c}{2d}r^2+\psi(0)$. Note that
  $\Delta\Phi=c=\Delta(\frac{c}{2d}\left|\cdot\right|^2)$ gives
  $\Delta(\Phi-\frac{c}{2d}\left|\cdot\right|^2)=0$. 
\end{proof}

\subsubsection{Completion of proof.}
To complete the proof of Theorem \ref{th:d=s+3} note that we can reinterpret
\eqref{eq:mustar3deuler} as Frostman conditions \eqref{eq:eulerlagrange}:
$\mu_{\mathrm{eq}}=\nu_R$ is seen as an equilibrium measure for kernel
$\widetilde K=K_{s+2}$ with external field $\widetilde V$ equal to $0$ on
$B_R$ and to $+\infty$ outside, connecting with Theorem \ref{th:riesz}.
\hfill\qed

\begin{remark}[Alternative motivation of the Frostman condition on
  $B_R$]\label{rm:invlap}
  Following \cite[Lemma 2.3]{GutlebCarrilloOlver2021Balls} or
  \cite[Sec.~4]{zbMATH06665921}, Riesz's formula \eqref{eq:riesz} with
  $d\geq2$, $d-2<s<d$, $R>0$ gives, using \eqref{eq:DeltaK} and
  \eqref{eq:DeltaKmu}, for $x\in\mathrm{int}(B_R)$,
  \begin{equation}\label{eq:rieszlaplacian1}
    \Delta\int_{|y|\leq R}
    \frac{\mathrm{d}y}{|x-y|^{s-2}(R^2-|y|^2)^{\frac{d-s}{2}}}
    =
    \int_{|y|\leq R}\frac{c_{d,s-2}\mathrm{d}y}{|x-y|^s(R^2-|y|^2)^{\frac{d-s}{2}}}
    =\frac{c_{d,s-2}\pi^{\frac{d}{2}+1}}{\Gamma(\frac{d}{2})\sin((d-s)\frac{\pi}{2})}.
  \end{equation}
  Now, inverting the Laplacian as in Lemma~\ref{le:liouville} and using
  Lemma~\ref{le:Phi} for continuity at the boundary, we get
  \begin{equation}\label{eq:rieszlaplacian2}
    \int_{|y|\leq R}
    \frac{\mathrm{d}y}{|x-y|^{s-2}(R^2-|y|^2)^{\frac{d-s}{2}}}
    =\frac{\pi^{\frac{d}{2}+1}}{\Gamma(\frac{d}{2})\sin((d-s)\frac{\pi}{2})}
    \Bigr(\frac{c_{d,s-2}|x|^2}{2d}+\frac{d-s}{2}R^2\Bigr).
  \end{equation}
  Replacing $s-2$ by $s$ gives, for $d\geq2$ and $d-4<s<d-2$, $R>0$, $x\in B_R$,
  \begin{equation}\label{eq:rieszgen1}
    \int_{|y|\leq R}
    \frac{\mathrm{d}y}{|x-y|^{s}(R^2-|y|^2)^{\frac{d-s}{2}-1}}
    =\frac{\pi^{\frac{d}{2}+1}}{\Gamma(\frac{d}{2})\sin((d-s)\frac{\pi}{2})}
    \Bigr(\frac{c_{d,s}}{2d}|x|^2-\frac{d-s-2}{2}R^2\Bigr).
  \end{equation}
  The left-hand side can be normalized using the fact that for $0<\beta<1$,
  \begin{equation}\label{eq:rieszgen2}
    Z_\beta:=
    \int_{|y|\leq R}
    \frac{\mathrm{d}y}{(R^2-|y|^2)^{\beta}}
    =|S_1|R^{d-2\beta}\int_0^1\frac{r^{d-1}\mathrm{d} r}{(1-r^2)^\beta}
    =\frac{R^{d-2\beta}\pi^{\frac{d}{2}}\Gamma(1-\beta)}{\Gamma(1-\beta+\frac{d}{2})}.
  \end{equation}
  Indeed, with $d-s=3$ and $\beta=\frac{d-s}{2}-1=\frac{1}{2}$, we get
  $Z_{\frac{1}{2}}=\frac{R^{s+2}\pi^{\frac{s+4}{2}}}{\Gamma(\frac{s+4}{2})}$,
  and \eqref{eq:rieszgen1} gives
  \begin{equation}\label{eq:rieszgen3}
    \frac{\Gamma(\frac{s+4}{2})}{R^{s+2}\pi^{\frac{s+4}{2}}}
    \int_{|y|\leq R}
    \frac{\mathrm{d}y}{|x-y|^{s}\sqrt{R^2-|y|^2}}
    +
    \frac{\Gamma(\frac{s+4}{2})\sqrt{\pi}c_{s+3,s}}{4R^{s+2}\Gamma(\frac{s+5}{2})}
    |x|^2
    = \frac{\Gamma(\frac{s+4}{2})\sqrt{\pi}}{2R^s\Gamma(\frac{s+3}{2})}.
  \end{equation}
  When $R$ is equal to the critical value \eqref{eq:R}, the prefactor of
  $|x|^2$ in \eqref{eq:rieszgen3} is equal to $\gamma$ and
  \eqref{eq:rieszgen3} becomes the Frostman condition on $B_R$ for Theorem
  \ref{th:d=s+3}. Note also that taking $(d,s)=(3,0)$ in \eqref{eq:rieszgen3}
  is allowed but produces a trivial kernel inside the integral in the
  left-hand side. It is also possible to take $d=3$ and $s\to0$ while keeping
  $s\neq0$, and use, for $x\neq0$,
  \begin{equation}\label{eq:K0}
    \lim_{\substack{s\to0\\s\neq0}}|s|^{-1}c_{3,s}=1
    \quad\text{and}\quad
    \lim_{\substack{s\to0\\s\neq0}}\Bigr(\frac{1}{s|x|^s}-\frac{1}{s}\Bigr)
    =\lim_{s\to0}\frac{|x|^{-s}-1}{s-0}=-\log|x|
  \end{equation}
  to recover the logarithmic kernel in this case. In another direction, note
  also that repeating the process that we used to get \eqref{eq:rieszgen1} to
  reach higher powers provides a family of generalizations of
  \eqref{eq:rieszgen1} involving Jacobi polynomials in the right-hand side,
  and even more generally hypergeometric ${}_2F_1$ functions, see for instance
  \cite{GutlebCarrilloOlver2021Balls,zbMATH06665921,zbMATH06324458}, producing
  potential extensions of Theorem \ref{th:d=s+3}.
\end{remark}

\begin{remark}[Hypergeometric formulas outside $B_R$]\label{rm:outside}
  As pointed out by the referee we can give hypergeometric formulas for the
  integrals \eqref{eq:co:riesz} and \eqref{eq:co:formula:sin} when
  $\lambda = |x|/R \geq 1$.
  More precisely, for $d\in\{2,3,4,\ldots\}$, $d-2<s<d$, $\lambda\geq1$,
  \begin{align}\label{eq:riesz:rhs}
    \int_{\mathbb{R}^d}
    \frac{|x-y|^{-s}}{(R^2-|y|^2)^{\frac{d-s}{2}}}\mathbf{1}_{|y|\leq R}\mathrm{d} y
    & = \frac{1}{\lambda^s}
    \frac{\pi^\frac{d}{2} \Gamma(\frac{s-d+2}{2})}{\Gamma(\frac{s+2}{2})}
   {}_2F_1\Bigr(\frac{s}{2},\frac{s-d+2}{2};\frac{s+2}{2};\frac{1}{\lambda^2}
   \Bigr)
  \end{align}
  and, using the function $\mathcal{S}_{d-3}$ defined in \eqref{eq:co:formula:sinS},
  \begin{equation}\label{eq:d=s+3:rhs}
    \int_0^1\mathcal{S}_{d-3}\Bigr(\frac{4\lambda r}{(\lambda+r)^2}\Bigr)
    \frac{(\lambda+r)^{3-d}r^{d-1}\mathrm{d} r}{\sqrt{1-r^2}}
    =
    \frac{1}{2\lambda^{d-3}}
    \frac{\Gamma(\frac{d-1}{2})^2\Gamma(\frac{d}{2})\Gamma(\frac{3}{2})}{\Gamma(d-1)\Gamma(\frac{d+1}{2})}
    {}_2F_1\Bigr(\frac{d-3}{2},-\frac{1}{2};\frac{d+1}{2};\frac{1}{\lambda^2}
    \Bigr).
  \end{equation}
  Indeed, to get \eqref{eq:riesz:rhs}, we start from the integral in the
  left hand side of \eqref{eq:co:riesz} by using, for $0\leq r<\lambda$,
  \begin{equation}
    {}_2F_1\Bigr(\frac{s}{4},\frac{s+2}{4};\frac{d}{2};
    \frac{4\lambda^2r^2}{(\lambda^2+r^2)^2}\Bigr)
    =\frac{(\lambda^2+r^2)^{\frac{s}{2}}}{\lambda^s}
    {}_2F_1\Bigr(\frac{s}{2},\frac{s-d+2}{2};\frac{d}{2};\frac{r^2}{\lambda^2}\Bigr),
  \end{equation}
  which comes from the quadratic transformation \eqref{eq:quad2F1} with
  $z = r^2/\lambda^2 < 1$. With this replacement the left-hand side of
  \eqref{eq:co:riesz} takes the form of the following Euler beta integral
  formula (see \cite{zbMATH02006424} and \cite[eq.\ (2.2.2)]{zbMATH01231230})
  \begin{equation}\label{eq:beta:alt}
    \int_0^\infty
    u^{\alpha-1}(1-u)^{\gamma-\alpha-1}
    {}_2F_1(a_1,a_2;b_1;tu)\mathrm{d} u
    =\frac{\Gamma(\alpha)\Gamma(\gamma-\alpha)}{\Gamma(\gamma)}
    {}_3F_2(a_1,a_2;\alpha;b_1;\gamma;t),
  \end{equation}
  which leads to \eqref{eq:riesz:rhs} after cancellation of upper and lower
  ${}_3F_2$ parameters. Similarly, to get \eqref{eq:d=s+3:rhs}, we start from
  the integral on the left-hand side of \eqref{eq:co:formula:sin} by using,
  for $0\leq r<\lambda$, the relation\footnote{The relation \eqref{E:quad2}
    can also be derived using Erdélyi~\cite[3.3(13)]{MR0058756} which is
    Whipple's relation between the Legendre functions of the first and second
    kinds, rewritten in terms of ${}_2F_1$. Both the first-kind $P_\nu^\mu$
    and the second-kind $Q_\nu^\mu$ have ${}_2F_1$ representations; see
    \cite[14.3.6 and 14.3.7]{NIST:DLMF}.}
  \begin{equation} \label{E:quad2}
    {}_2F_1\Bigr(\frac{s}{2},\frac{s+2}{2};s+2;\frac{4\lambda
      r}{(\lambda+r)^2}\Bigr)%
    =\frac{(\lambda+r)^s}{\lambda^s}
    {}_2F_1\Bigr(\frac{s}{2},-\frac{1}{2};\frac{s+3}{2};\frac{r^2}{\lambda^2}\Bigr),
  \end{equation}
  which comes from the quadratic transformation \eqref{eq:quad2F14E} with
  $z = r/\lambda < 1$. This leads to \eqref{eq:d=s+3:rhs} via
  \eqref{eq:beta:alt}.
\end{remark}

\subsection{Proof of Corollary \ref{co:formulas}}
\label{se:co:formulas}

The function $\varphi$ defined in \eqref{Eq:Phidef} is continuous on $[0,1]$
and differentiable on $(0,1)$. From the proof of Theorem \ref{th:d=s+3}, the
Frostman condition states that $\varphi$ is constant and equal to $\varphi(0)$
on $\lambda\in[0,1]$, namely $\varphi'(\lambda)=0$ for $\lambda\in(0,1)$. Now
the formula \eqref{eq:co:formula:sin} in Corollary \ref{co:formulas} comes
from the combination of equation $\varphi(\lambda)=\varphi(0)$ together with
the formulas for $\varphi$ provided by Lemmas \ref{le:varphi} and
\ref{le:landend} of Appendix~\ref{se:key}. Note that \eqref{eq:co:formula:sin}
is trivial when $d=3$. The formula \eqref{eq:co:formula:sinF} is obtained from
\eqref{eq:co:formula:sin} by using \eqref{eq:co:formula:sinS} and the Legendre
duplication formula \eqref{eq:duplication}.%

\subsection{Proof of Corollary \ref{co:formulas:more}}
\label{se:co:formulas:more}

Let us keep the notation used in Appendix \ref{se:key}. First of all, the
formulas (\ref{eq:co:formulas:log}--\ref{eq:co:formulas:log'}) in Corollary
\ref{co:formulas:more} come from the Frostman condition
$\varphi(\lambda)=\varphi(0)$ and its reformulation $\varphi'(\lambda)=0$, and
the formula for $\varphi$ provided by Lemma \ref{le:varphi}.

The formula \eqref{eq:co:formulas:E} in Corollary \ref{co:formulas:more} is
obtained by further reformulating $\varphi$ when $(d,s)=(2,-1)$ in terms of
special functions using Lemma \ref{le:varphi:d2s-1} below. The formula for
$\varphi'$ provided by this lemma gives
\begin{equation}\label{eq:co:formulas:EK}
  \int_0^1
  \left[
    (\lambda+r)
    E\left(\frac{4\lambda r}{(r+\lambda)^2}\right)
    +
    (\lambda-r)
    K\left(\frac{4\lambda r}{(r+\lambda)^2}\right)
  \right]
  \frac{r\,\mathrm{d}r}{\sqrt{1-r^2}}
  =\frac{\pi^2}{4}\lambda^2.
\end{equation}
Next, following
\cite{doi:10.1098/rstl.1775.0028,doi:10.1098/rstl.1771.0037,MR966232}, the
Landen transform for $E$ and $K$ gives, for $z\in[-1,1]$,
\begin{equation}\label{eq:landenEK}
  K\Bigr(\frac{4z}{(1+z)^2}\Bigr)=(1+z)K(z^2)
  \quad\text{and}\quad
  E\Bigr(\frac{4z}{(1+z)^2}\Bigr)=\frac{2}{1+z}E(z^2)-(1-z)K(z^2).
\end{equation}
Now, the formula \eqref{eq:co:formulas:K} of Corollary \ref{co:formulas} comes
by combining \eqref{eq:co:formulas:E} and \eqref{eq:co:formulas:EK} with
$z=\frac{r}{\lambda}$.

\appendix

\section{Useful tools}

\label{appA}

Let us recall the Euler reflection formula for the Gamma function, valid for $z\not\in\{-1,-2,\ldots\}$,
\begin{equation}\label{eq:reflection}
  \Gamma(z)\Gamma(1-z)=\frac{\pi}{\sin(\pi z)},
\end{equation}
and the Legendre duplication formula, valid for
$2z\not\in\{-0,-1,-2,-3,\ldots\}$,
\begin{equation}\label{eq:duplication}
  \sqrt{\pi}\Gamma(2z)=2^{2z-1}\Gamma(z)\Gamma\Bigr(z+\frac{1}{2}\Bigr).
\end{equation}

\subsection{Hypergeometric Identities}
\label{appA.1}

\begin{itemize}
\item The hypergeometric function ${}_2F_1$ can be written as (see \cite[(15.2(i))]{NIST:DLMF})
  \begin{equation}\label{eq:2F1def}
    {}_2F_1\left(a, b; c; z\right) :=
    \frac{\Gamma(c)}{\Gamma(a)\Gamma(b)}
    \sum_{k=0}^\infty \frac{\Gamma(a+k) \Gamma(b+k)}{\Gamma(c+k)} \; \frac{z^k}{k!}.
  \end{equation}
\item If $\Re(c - a - b) > 0$ then \eqref{eq:2F1def} converges
  absolutely for $|z| \leq 1$ and (see \cite[(15.4.20)]{NIST:DLMF})
  \begin{equation}\label{eq:2F1z1}
    {}_2F_1\left(a, b; c; 1\right)  =
    \frac{\Gamma(c) \Gamma(c-a-b)}{\Gamma(c-a)\Gamma(c-b)} .
  \end{equation}
\item If $c = a + b$ then (\cite[(15.4.21)]{NIST:DLMF}):
  \begin{equation}\label{eq:2F1z1a}
    \lim_{z\to 1^-} \frac{{}_2F_1\left(a, b; a+b; z\right)}{-\log(1 - z)}  =
    \frac{\Gamma(a+b)}{\Gamma(a)\Gamma(b)} .
  \end{equation}
\item Quadratic transformation (see \cite[(15.8.13)]{NIST:DLMF} %
  or \cite[2.11(4)]{MR0058756})%
  : if $|\mathrm{phase}(1-z)| < \pi$ then
  \begin{equation} \label{eq:2F1Q}
    {}_2F_1\left(\frac{a}{2}, \frac{1}{2} + \frac{a}{2}; \frac{1}{2} + b; \frac{z^2}{(2-z)^2}\right) = \left(1 - \frac{z}{2}\right)^{a}\,{}_2F_1\left(a, b; 2b; z\right).
  \end{equation}
\item Quadratic transformation (\cite[2.11(34)]{MR0058756}): if $0\leq z\leq1$
  then
    \begin{equation}\label{eq:quad2F1}
      {}_2F_1\Bigr(\frac{a}{2},\frac{a+1}{2};a-b+1;\frac{4z}{(1+z)^2}\Bigr)
      =(1+z)^{a}\,{}_2F_1(a,b;a-b+1;z)
    \end{equation}
    (there is a typo in \cite[2.11(34)]{MR0058756}: $a-b-1$ has been corrected
    here to $a-b+1$).%
  \item Quadratic transformation (see \cite[2.11(5)]{MR0058756}): if
    $0\leq z\leq1$ then
  \begin{equation}\label{eq:quad2F14E}
    {}_2F_1\Bigr(a,b;2b;\frac{4z}{(1+z)^2}\Bigr)
    =(1+z)^{2a}\,{}_2F_1\Bigr(a,a+\frac{1}{2}-b;b+\frac{1}{2};z^2\Bigr).
  \end{equation}
\item Derivative formula (see \cite[(15.5.1)]{NIST:DLMF}):
  \begin{equation}\label{eq:2F1deriv}
    \frac{\mathrm{d}}{\mathrm{d} z} {}_2F_1\left(a, b; c; z\right) =
    \left(\frac{ab}{c}\right) {}_2F_1\left(a+1, b+1; c+1; z\right).
  \end{equation}
\item Euler integral formula
    (see \cite[p.~4-5]{MR0185155} and \cite[(15.6.1)]{NIST:DLMF}):
    \begin{equation}\label{eq:2F1Int}
      {}_2F_1\left(a, b; c; z\right) = \frac{\Gamma(c)}{\Gamma(b)\Gamma(c-b)}
      \int_0^1 \frac{u^{b-1} (1 - u)^{c-b-1}}{(1 - z u)^a} \mathrm{d} u,
    \end{equation}
    provided that $\Re(b)>0$, $\Re(c) > 0$, and $|\mbox{phase}(1-z)| < \pi$.
\end{itemize}

\subsection{Funk\,--\,Hecke formula}
\label{appA.2}

Let $d\geq2$ and $\sigma_{S_1}$ denote the uniform probability measure on the unit
centered sphere $S_1=\{x\in\mathbb{R}^d:|x|=1\}$. Then, for $z\in\mathbb{R}^d$
with $|z|=1$,
\begin{equation}\label{eq:funkhecke}
  \int_{S_1}
  f(z\cdot x)\sigma_{S_1}(\mathrm{d} x)
  =\tau_{d-1}\int_0^\pi
  f(\cos(\theta))\sin^{d-2}(\theta)\mathrm{d}\theta
  =\tau_{d-1}\int_{-1}^1f(t)(1-t^2)^{\frac{d-3}{2}}\mathrm{d} t
\end{equation}
where
\begin{equation}\label{eq:taud-1}
  \tau_{d-1}
  :=\left(\int_0^\pi\sin^{d-2}(\theta)\mathrm{d}\theta\right)^{-1}
  =\left(\int_{-1}^1(1-t^2)^{\frac{d-3}{2}}\mathrm{d} t\right)^{-1}
  =\frac{\Gamma(\frac{d}{2})}{\Gamma(\frac{1}{2})\Gamma(\frac{d-1}{2})}.
\end{equation}

The Funk\,--\,Hecke formula \eqref{eq:funkhecke} is a useful tool to reduce
multivariate integrals into univariate integrals. It gives the projection on
any diameter of the uniform law on the sphere. If $X$ is a random vector in
$\mathbb{R}^d$ uniformly distributed on $S_1$ then for $z\in S_1$, the law of
$z\cdot X$ has density
$\tau_{d-1}(1-t^2)^{\frac{d-3}{2}}\mathbf{1}_{t\in[-1,1]}$. This is an arcsine
law when $d=2$, a uniform law when $d=3$, a semicircle law when $d=4$, and
more generally, for an arbitrary $d\geq2$, the image by the map
$u\mapsto\sqrt{u}$ of the beta law $\mathrm{Beta}(\frac{1}{2},\frac{d-1}{2})$.
We refer to \cite[p,~18]{MR0199449} or \cite[Eq.~(5.1.9)~p.~197]{MR3970999}
for a proof.

\subsection{Euler\,--\,Lagrange characterization of equilibrium measure (Frostman conditions)}
\label{appA.3}

For $\mu\in\mathcal{M}_1$ such that
$K_s(x)\mathbf{1}_{|x|>1}\mathbf{1}_{s\leq0}\in L^1(\mu)$, we define the
$s$-Riesz potential at point $x\in\mathbb{R}^d$ by
\begin{equation}\label{eq:Umu}
  U^\mu(x):=(K_s*\mu)(x)=\int K_s(x-y)\mu(\mathrm{d}y)\in(-\infty,+\infty].
\end{equation}
The Euler\,--\,Lagrange characterization of the equilibrium measure $\mu_{\mathrm{eq}}$,
also known as \emph{Frostman conditions} in potential theory, states that a
necessary and sufficient condition for such an element $\mu$ of
$\mathcal{M}_1$ to be an equilibrium measure is that for some finite constant
$c$ we have (see, for example, \cite{MR0350027})
\begin{equation}\label{eq:eulerlagrange}
  U^{\mu} + V
  \begin{cases}
   \leq c & \text{ on the support of $\mu$} \\
   \geq c & \text{ quasi-everywhere on $\mathbb{R}^d$}
  \end{cases} ;
\end{equation}
by ``quasi-everywhere'' we mean except on a set for which every probability
measure supported on it has infinite energy. This condition holds everywhere
when $V$ is continuous. It is customary to say that $c$ is the \emph{modified
  Robin constant} and we have
$c=\int U^{\mu_{\mathrm{eq}}}\mathrm{d} \mu_{\mathrm{eq}}+\int V\mathrm{d} \mu_{\mathrm{eq}}=\mathrm{I}(\mu_{\mathrm{eq}})-\int V\mathrm{d} \mu_{\mathrm{eq}}$.

\subsection{Integrability and regularity of Riesz potentials}
\label{appA.4}

The following Lemma summarizes key regularity properties of the Riesz kernel,
some of which are classical. We give a proof for the reader's convenience. On
this topic, we also refer to the works of Mizuta such as
\cite{MR1018814,MR1211771,MR1428685}.

\begin{lemma}[Integrability and regularity of Riesz potentials]
\label{le:rieszpot} 
  \begin{enumerate}[(iii)]
  \item[(i)] $K_s\in L^1_{\mathrm{loc}}(\mathbb{R}^d,\mathrm{d} x)$
    if and only if $s=0$ or $s\neq0$ and $s<d$.
  \item[(ii)] If $s<d-2$ then, in the sense of distributions, and in the sense of
    functions on $\{x\in\mathbb{R}^d:x\neq0\}$,
    \begin{equation}\label{eq:DeltaK}
      \Delta K_s=-c_{d,s}K_{s+2}
      \quad\text{where}\quad
      c_{d,s}:=
      \begin{cases}
        |s|(d-2-s)&\text{if $s\neq0$}\\
        d-2&\text{if $s=0$}
      \end{cases}.
    \end{equation}
  \item[(iii)] Suppose that $s=0$ or $s\neq0$ and $s<d$. Let $\mu$ be a compactly
    supported probability measure on $\mathbb{R}^d$. Then the following
    function is well defined and belongs to
    $L^1_{\mathrm{loc}}(\mathbb{R}^d,\mathrm{d} x)$:
    \begin{equation}\label{eq:Kmu}
      x\in\mathbb{R}^d\mapsto
      (K_s*\mu)(x):=\int K_s(x-y)\mu(\mathrm{d} y).
    \end{equation}
    Moreover, in the sense of distributions,
    \begin{equation}\label{eq:DeltaKmu}
      \Delta (K_s*\mu)=(\Delta K_s)*\mu=-c_{d,s}K_{s+2}*\mu.
    \end{equation}
  \item[(iv)] Suppose that $s=0$ or $s\neq0$ and $s<d$. If $\mu$ is a compactly
    supported probability measure on $\mathbb{R}^d$ such that
    $\mu(\mathrm{d} x)=f(x)\mathrm{d} x$, $f\in L^p(\mathbb{R}^d,\mathrm{d} x)$, and $p>d/(d-s)$,
    then $K_s*\mu$ is continuous on $\mathbb{R}^d$.
  \end{enumerate}
\end{lemma}

Note that $K_{s+2}\in L^1_{\mathrm{loc}}(\mathbb{R}^d,\mathrm{d} x)$ implies $s+2<d$;
in other words $s<d-2$. Furthermore, the condition $s<d-2$ is sharp for
\eqref{eq:DeltaK}, indeed; in the sense of distributions, we have
$\Delta K_{d-2}=-c_d\delta_0$ (Coulomb kernel). This suggests defining
$K_d:=\delta_0$ to make the formula \eqref{eq:DeltaK} valid for the critical
case $s=d-2$, provided that we also set $c_{d,d-2}:=c_d$.

We remark that \eqref{eq:DeltaK} is a special case of \eqref{eq:DeltaKmu} which
corresponds to taking $\mu=\delta_0$ and that
\eqref{eq:DeltaKmu} goes beyond \cite[Eq.~(7)~p.~118]{MR0290095} and
\cite[Eq.~(85)~p.~136]{MR1790156}.
Note also that the distribution $\Delta\mu$ equals the convolution
$(\Delta\delta_0)*\mu$, see \cite[end of Ch.~VI, Sec.~3; notably eq.~(VI, 3;
34--35)]{MR0209834}. From this point of view, it follows that
\eqref{eq:DeltaKmu} is a consequence of the associative law for convolution of
three distributions, two of which have compact support, see \cite[Ch.~VI,
Sec.~3, Th.~VII]{MR0209834} and \cite[Lemma 0.6]{MR0350027}. We give however a
direct short proof of \eqref{eq:DeltaKmu} below.

\begin{proof}[Proof of Lemma \ref{le:rieszpot}]
  \emph{Proof of (i).} It suffices to check local integrability in the
  neighborhood of the origin. We have
  \[
    \int_{|x|\leq1}|K_s(x)|\mathrm{d} x
    =
    \begin{cases}
      \displaystyle2\pi\int_0^1r^{-s+d-1}\mathrm{d} r<\infty
      &\text{ if $s\neq0$ and $d>s$}\\[.5em]
      \displaystyle2\pi\int_0^1\log(r)r^{d-1}\mathrm{d} r<\infty
      &\text{ if $s=0$}
    \end{cases}.
  \]

  \emph{Proof of (ii).} On $\mathbb{R}^d\setminus\{0\}$, the function $K_s$ is
  $\mathcal{C}^\infty$ and a computation reveals that
  \[
    \Delta K_s=-c_{d,s}K_{s+2}.
  \]
  It follows that this equality also holds in the sense of distributions for
  test functions supported away from the origin. For general test functions,
  we proceed by integration by parts outside a centered ball of small radius.
  Namely, let $\varphi$ be a compactly supported $\mathcal{C}^\infty$ test
  function, and let $\varepsilon>0$. By the Green integration by parts formula
  for the open set $\{x\in\mathbb{R}^d:|x|>\varepsilon\}$, denoting
  $n(x)=-x|x|^{-1}$ the inner unit normal vector to the sphere
  $\{x\in\mathbb{R}^d:|x|=\varepsilon\}$ at the point $x$,
  \begin{multline*}
    \int_{|x|\geq\varepsilon}\Delta\varphi(x)K_s(x)\mathrm{d} x
    -\int_{|x|\geq\varepsilon}\varphi(x)\Delta K_s(x)\mathrm{d} x\\
    =\int_{|x|=\varepsilon}K_s(x)\nabla\varphi(x)
    \cdot n(x)\mathrm{d} \sigma_\varepsilon(x)
    -\int_{|x|=\varepsilon}\varphi(x)\nabla K_s(x)
    \cdot n(x)\mathrm{d} \sigma_\varepsilon(x).
  \end{multline*}

  If $s\neq0$ and $d>s+1$ then
  \begin{align*}
    \int_{|x|=\varepsilon}K_s(x)\nabla\varphi(x)
    \cdot n(x)\mathrm{d} \sigma_\varepsilon(x)
    &=\varepsilon^{-s}\int_{|x|=\varepsilon}\nabla\varphi(x)
    \cdot n(x)\mathrm{d} \sigma_\varepsilon(x)\\
    &=\varepsilon^{-s}O(\varepsilon^{d-1})
    =O(\varepsilon^{-s+d-1})
    =o_{\varepsilon\to0^+}(1),
  \end{align*}
  while, using $\nabla K_s(x)=-sx|x|^{-(s+2)}=-sxK_{s+2}(x)$ and
  $x\cdot n_x=-|x|$, if $d>s+2$,
  \begin{align*}
    \int_{|x|=\varepsilon}\varphi(x)\nabla K_s(x)
    \cdot n(x)\mathrm{d} \sigma_\varepsilon(x)
    &=
    s\varepsilon^{1-(s+2)}
    \int_{|x|=\varepsilon}\varphi(x)\mathrm{d} \sigma_\varepsilon(x)\\
    &=s\varepsilon^{1-(s+2)}O(\varepsilon^{d-1})
    =O(\varepsilon^{d-(s+2)})
    =o_{\varepsilon\to0^+}(1).
  \end{align*}
  Finally a careful analysis reveals that the conditions on $d$ are the same
  in the case $s=0$.

  \emph{Proof of (iii).} If $s<0$, then $K_s\leq0$, and hence $K_s*\mu$ is
  well defined and takes its values in $[-\infty,0]$. Similarly, if $s>0$,
  then $K_s\geq0$, and hence $K_s*\mu$ is well defined and takes its values in
  $[0,+\infty]$. If $s=0$ then $K_0\mathbf{1}_{\left|\cdot\right|\leq1}\geq0$
  while
  $\sup_{\mathbb{R}^d}K_0\mathbf{1}_{\left|\cdot\right|\geq1}/\log(1+\left|\cdot\right|)<\infty$,
  hence $K_0*\mu$ is well defined and takes its values in $(-\infty,+\infty]$.
  Next, by the Fubini\,--\,Tonelli theorem, for $R>0$, using (i) and
  the compactness of support of $\mu$ (note that this can be weakened into
  integrability of $\log(1+\left|\cdot\right|)\mathbf{1}_{s=0}$),
  \[
    \iint|K_s(x-y)|\mathbf{1}_{|x|\leq R}\mu(\mathrm{d} y)\mathrm{d} x
    =\int\Bigr(\int|K_s(x)|\mathbf{1}_{|x+y|\leq
      R}\mathrm{d} x\Bigr)\mu(\mathrm{d} y)
    <\infty.
  \]
  It follows that $K_s*\mu$ belongs to
  $L^1_{\mathrm{loc}}(\mathbb{R}^d,\mathrm{d} x)$.

  For the differentiability, let $\varphi:\mathbb{R}^d\to\mathbb{R}$ be a
  $\mathcal{C}^\infty$ and compactly supported test function. By the
  Fubini\,--\,Tonelli theorem, the Green integration by parts formula, and
  (ii), we have
  \begin{align*}
    \int(K_s*\mu)(x)\Delta\varphi(x)\mathrm{d} x
    &=
    \int\Bigr(\int
      K_s(x-y)\mu(\mathrm{d} y)\Bigr)\Delta\varphi(x)\mathrm{d} x\\
    &=
      \int\Bigr(\int K_s(x-y)\Delta\varphi(x)\mathrm{d} x\Bigr)\mu(\mathrm{d} y)\\
    &=-c_{d,s}
    \int
      \Bigr( \int
      K_{s+2}(x-y)\varphi(x)\mathrm{d} x\Bigr) 
      \mu(\mathrm{d} y)\\
    &=-c_{d,s}\int\varphi(x)\Bigr(\int
      K_{s+2}(x-y)
      \mu(\mathrm{d} y)
      \Bigr)\mathrm{d} x\\
    &=-c_{d,s}\int\varphi(x)(K_{s+2}*\mu)(x)\mathrm{d} x.
  \end{align*}

  \emph{Proof of (iv).} For the continuity, we follow closely the cutoff
  argument used in \cite[Lem.~4.3]{MR3262506}, see also
  \cite[Th.~1]{MR1018814}, \cite{MR1211771}, and \cite[Sec.~5.3]{MR1428685}.
  Namely, let us consider first the case $s>0$. For $n\geq1$ and
  $x\in\mathbb{R}^d$, let us define
  \[
    R_n(x):=\int f(y)K_s(x-y)\mathbf{1}_{|K_s(x-y)|\leq n}\mathrm{d} y
  \]
  and
  \[
    T_n(x):=(K_s*\mu)(x)-R_n(x)
    =\int f(y)K_s(x-y)\mathbf{1}_{|K_s(x-y)|\geq n}\mathrm{d} y.
  \]
  By the dominated convergence theorem, $R_n$ is continuous on $\mathbb{R}^d$.
  Let us show now that $\lim_{n\to\infty}T_n=0$ uniformly on compact subsets,
  which will prove the continuity of $K_s*\mu$. Let $q:=p/(p-1)$ be the Hölder
  conjugate exponent of $p$. Now, by the Hölder inequality, using the fact
  that $K_s=\left|\cdot\right|^{-s}$, $s>0$,
  \[
    0\leq T_n(x)
    =\int f(y)\frac{\mathbf{1}_{|x-y|\leq n^{-1/s}}}{|x-y|^s}\mathrm{d} y
    \leq\left\Vert f\right\Vert_{p,B(x,1)}\varepsilon_n^{1/q}
  \]
  where $B(x,r):=\{x\in\mathbb{R}^d:|x|\leq r\}$ is the closed centered ball
  of radius $r$, where $\left\Vert\cdot\right\Vert_{p,C}$ denotes the $L^p$
  norm with respect to the trace of the Lebesgue measure on $C$, where
  \[
    \varepsilon_n:=|\mathbb{S}^{d-1}|
    \int_0^{n^{-1/s}}\frac{\mathrm{d} r}{r^{qs-d+1}},
  \]
  and where $|\mathbb{S}^{d-1}|$ is the surface area of the unit sphere
  $\{x\in\mathbb{R}^d:|x|=1\}$. The condition $p>d/(d-s)$, which is equivalent
  to $qs-d+1<1$, ensures that $\varepsilon_n$ is finite for all $n$ and that
  $\lim_{n\to\infty}\varepsilon_n=0$. Hence, if $C\subset\mathbb{R}^d$ is a
  compact set, then, denoting
  $C_1:=\{x\in\mathbb{R}^d:\mathrm{dist}(x,C)\leq1\}$, we have
  \[
    \sup_{x\in C}|T_n(x)|
    \leq\left\Vert f\right\Vert_{p,K_1}\varepsilon_n^{1/q}
    \underset{n\to\infty}{\longrightarrow}0,
  \]
  which completes the proof of the continuity of $K_s*\mu$. The case $s<0$ is
  entirely similar up to a sign. It remains to examine the case $s=0$. Let us
  write $K_0=K_0^+-K_0^-$ with $K_0^\pm\geq0$, namely
  $K_0^+=-\log\left|\cdot\right|\mathbf{1}_{\left|\cdot\right|\leq1}$ and
  $K_0^-=\log\left|\cdot\right|\mathbf{1}_{\left|\cdot\right|>1}$. To
  establish the continuity of $K_0^+*\mu$ we write
  \[
    0\leq T_n^+(x)
    :=\int f(y)
      \log\frac{1}{|x-y|}
      \mathbf{1}_{|x-y|\leq1}\mathbf{1}_{|x-y|\leq \mathrm{e}^{-n}}\mathrm{d} y\\
      \leq\left\Vert f\right\Vert_{p,B(x,1)}
      (\varepsilon_n^+)^{1/q}
  \]
  where
  \[
    \varepsilon_n^+
    =-|\mathbb{S}^{d-1}|\int_0^{\mathrm{e}^{-n}}r^{d-1}\log(r)\mathrm{d} r
    \underset{n\to\infty}{\longrightarrow}0.
  \]
  On the other hand, the continuity of $K_0^-*\mu$ follows from that of
  $K_0^-$. Hence $K_0*\mu$ is continuous.
\end{proof}

\section{Key formulas for potential plus external field}
\label{se:key}
\label{appB}

Let $d$, $s$, $\mu_{\mathrm{eq}}=\nu_R$, $R$, and $\sigma_{S_1}$ be as in Theorem
\ref{th:d=s+3}. For $x\in\mathbb{R}^d$, the quantity
$\Phi(x):=(K_s*\mu_{\mathrm{eq}})(x)+\gamma\left|x\right|^2$ depends only on
$\lambda:=|x|/R$ and we can define
\begin{equation}\label{eq:varphi}
  \varphi(\lambda)
  :=
  \Phi(x)
  =U^{\mu_{\mathrm{eq}}}(x)+\gamma\left|x\right|^2
  =U^{\mu_{\mathrm{eq}}}(x)+\gamma R^2\lambda^2.
\end{equation}

The following lemmas provide key formulas for the potential plus external
field $\varphi$.

\begin{lemma}[Integral formula for potential]\label{le:varphi}
  For $\lambda\geq0$, denoting
  $\displaystyle c_d := \frac{2
    \,\mathrm{sign}(d-3)\Gamma(\frac{d+1}{2})}{\pi\Gamma(\frac{d-1}{2})}$,
  \[
    \varphi(\lambda) =
    \begin{cases}
      \displaystyle
      \frac{1}{R^{d-3}}\left(
        c_d \int_0^1 \Bigr(\int_0^{\pi}
        \frac{\sin^{d-2}(\theta)}
        {(\lambda^2-2r\lambda\cos(\theta)+r^2)^{\frac{d-3}{2}}}\mathrm{d} \theta\Bigr)
        \frac{r^{d-1}\,\mathrm{d}r}{\sqrt{1-r^2}}+\gamma R^{d-1}\lambda^2\right)
      &\text{if $d\neq 3$}\\[1em]
      \displaystyle
      -\int_0^1\!\frac{(\lambda +r)^2 \log(\lambda+r)-(\lambda-r)^2\log|\lambda-r|}{\pi\lambda} \frac{r\,\mathrm{d}r}{\sqrt{1-r^2}}-\log R+\frac{1}{2}+\gamma R^2\lambda^2
      &\text{if $d = 3$}
    \end{cases}.
  \]
\end{lemma}

Note that $\gamma R^{d-1}$ does not depend on $\gamma$.

\begin{proof}
  By the Funk\,--\,Hecke formula \eqref{eq:funkhecke}, for $x\in\mathbb{R}^d$,
  $s\neq0$, with
  $C_d:= \mathrm{sign}(s)\frac{\Gamma(\frac{d+1}{2})}{\pi^{\frac{d+1}{2}}}$,
  \begin{align}
    U^{\mu_{\mathrm{eq}}}(x)
    &=C_d
      \int_{|y|\leq1}\frac{\mathrm{d} y}{|x-R y|^{s}\sqrt{1-|y|^2}}\nonumber\\
    &=C_d|S_1|
      \int_0^1\int_{S_1}\frac{\sigma_{S_1}(\mathrm{d} y)
      r^{d-1}\,\mathrm{d}r}{(|x|^2-2rR\langle
      x,y\rangle+r^2R^2)^{\frac{s}{2}}\sqrt{1-r^2}}\nonumber\\
    &=\frac{C_d|S_1|\tau_{d-1}}{R^s}
      \int_0^1\Bigr(\int_{-1}^1\frac{(1-t^2)^{\frac{s}{2}}}{(\lambda^2-2r\lambda
      t+r^2)^{\frac{s}{2}}}\mathrm{d} t\Bigr)\frac{r^{d-1}}{\sqrt{1-r^2}}\,\mathrm{d}r\nonumber\\
    &=\frac{2\mathrm{sign}(s)\Gamma(\frac{d+1}{2})}{\pi\Gamma(\frac{d-1}{2})R^{d-3}}
      \int_0^1\Bigr(\int_0^{\pi}\frac{\sin^{s+1}(\theta)}{(\lambda^2-2r\lambda\cos(\theta)+r^2)^{\frac{s}{2}}}\mathrm{d} \theta\Bigr)\frac{r^{d-1}}{\sqrt{1-r^2}}\,\mathrm{d}r,
  \end{align}
  while if $s=0$ ($d=3$),
  \begin{align*}
    U^{\mu_{\mathrm{eq}}}(x)
    &=-\frac{1}{\pi^2}\int_{|y|\leq 1}\frac{\log\left|x-Ry\right|}{\sqrt{1-|y|^2}}\mathrm{d} y\\
    &=-\frac{2}{\pi}\int_0^1\Bigr(\int_{\mathbb{S}^2}\log\bigr(\lambda^2R^2-2rR\langle
      x,y\rangle+r^2R^2\bigr)\sigma_{S_1}(\mathrm{d} y)\Bigr)\frac{r^2}{\sqrt{1-r^2}}\,\mathrm{d}r\\
    &=-\frac{1}{\pi}\int_0^1\Bigr(4\log R+\int_{-1}^1\log(\lambda^2-2r\lambda t+r^2)\mathrm{d} t\Bigr)\frac{r^2}{\sqrt{1-r^2}}\,\mathrm{d}r.
  \end{align*}
  Finally we observe that
  \[
  \int_{-1}^1\log(\lambda^2-2r\lambda t+r^2)\mathrm{d} t
  =\frac{(\lambda
    +r)^2\log(\lambda+r)-(\lambda-r)^2\log|\lambda-r|}{r\lambda}-2.
\]
\end{proof}

\begin{lemma}[Landen transform and a special function]\label{le:landend}
  For $d\in\{2,3,\ldots\}$, $\lambda\geq0$, and $r\in[0,1]$,
  \[
    \int_0^{\pi}\frac{\sin^{d-2}(\theta)}{(\lambda^2-2r\lambda\cos(\theta)+r^2)^{\frac{d-3}{2}}}\mathrm{d}\theta
    =
    \frac{2^{d-1}}{(\lambda+r)^{d-3}}\mathcal{S}_{d-3}\left(\frac{4\lambda
      r}{(\lambda+r)^2}\right),
  \]
  where for $z\in[0,1]$,
  \[
    \mathcal{S}_s(z):=
    \int_0^{\frac{\pi}{2}}\frac{\sin^{s+1}(\alpha)\cos^{s+1}(\alpha)}{(1-z\sin^2(\alpha))^{\frac{s}{2}}}\mathrm{d}\alpha
    = \int_0^1\frac{t^{s+1}(1-t^2)^{\frac{s}{2}}}{(1-zt^2)^{\frac{s}{2}}}\mathrm{d} t
    = \frac{\Gamma(\frac{s}{2}+1)^2}{2\Gamma(s+2)}
    \,{}_2F_1\Bigr(\frac{s}{2}+1, \frac{s}{2};s+2;z\Bigr).
  \]
\end{lemma}

\begin{proof}
  We set $\rho_1:=\frac{2\lambda r}{\lambda^2+r^2}$,
  $\rho_2:=\frac{2\rho_1}{1+\rho_1}=\frac{4\lambda r}{(\lambda+r)^2}$, which
  gives $(\lambda^2+r^2)(1+\rho_1)=(\lambda+r)^2$. Using the change of
  variable $\theta=2\alpha$, and $\cos(\theta)=1-2\sin^2(\alpha)$,
  $\sin(\theta)=2\sin(\alpha)\cos(\alpha)$, we get
  \begin{align*}
    \int_0^{\pi}\frac{\sin^{s+1}(\theta)}{(\lambda^2-2r\lambda\cos(\theta)+r^2)^{\frac{s}{2}}}\mathrm{d}
    \theta
    &=\frac{1}{(\lambda^2+r^2)^{\frac{s}{2}}}
      \int_0^{\pi}\frac{\sin^{s+1}(\theta)}{(1-\rho_1\cos(\theta))^{\frac{s}{2}}}\mathrm{d}
      \theta\\
    &=\frac{2^{s+2}}{(\lambda^2+r^2)^{\frac{s}{2}}}
      \int_0^{\frac{\pi}{2}}\frac{\sin^{s+1}(\alpha)\cos^{s+1}(\alpha)}{(1-\rho_1(1-2\sin^2(\alpha)))^{\frac{s}{2}}}\mathrm{d}
      \alpha\\
    &=\frac{2^{s+2}}{(\lambda^2+r^2)^{\frac{s}{2}}(1-\rho_1)^{\frac{s}{2}}}
      \int_0^{\frac{\pi}{2}}\frac{\sin^{s+1}(\alpha)\cos^{s+1}(\alpha)}{(1+\frac{2\rho_1}{1-\rho_1}\sin^2(\alpha))^{\frac{s}{2}}}\mathrm{d}
      \alpha\\
    &=\frac{2^{s+2}}{(\lambda+r)^s}
      \frac{(1+\rho_1)^{\frac{s}{2}}}{(1-\rho_1)^{\frac{s}{2}}}
      \mathcal{S}_s\Bigr(-\frac{2\rho_1}{1-\rho_1}\Bigr).
  \end{align*}
  But for $z\in[0,1]$,
  \begin{equation*}
    \mathcal{S}_s(-z) = \int_0^{\frac{\pi}{2}}
    \frac{\sin^{s+1}(\alpha)\cos^{s+1}(\alpha)}
    {(1+z\sin^2(\alpha))^{\frac{s}{2}}} \mathrm{d}\alpha =
    \frac{1}{(1+z)^{\frac{s}{2}}}
    \int_0^{\frac{\pi}{2}}\frac{\sin^{s+1}(\alpha)\cos^{s+1}(\alpha)}
    {(1-\frac{z}{1+z}\cos^2(\alpha))^{\frac{s}{2}}} \mathrm{d}\alpha =
    \frac{1}{(1+z)^{\frac{s}{2}}}\mathcal{S}_s\Bigr(\frac{z}{1+z}\Bigr).
  \end{equation*}
  In particular, with $z=\frac{2\rho_1}{1-\rho_1}$, we get
  $1+z=\frac{1+\rho_1}{1-\rho_1}$ and
  $\frac{z}{1+z}=\frac{2\rho_2}{1+\rho_1}=\rho_2$; therefore
  \[
    \mathcal{S}_s\Bigr(-\frac{2\rho_1}{1-\rho_1}\Bigr)
    =\frac{(1-\rho_1)^{\frac{s}{2}}}{(1+\rho_1)^{\frac{s}{2}}}\mathcal{S}_s(\rho_2),
  \]
  hence the desired integral formula in terms of $\mathcal{S}_{d-3}$.
  Finally the hypergeometric formula for $\mathcal{S}_s$ follows from
  $\mathcal{S}_s(z)=\frac{1}{2}\int_0^1u^{\frac{s}{2}}(1-u)^{\frac{s}{2}}(1-zu)^{-\frac{s}{2}}\mathrm{d}
  u$ and the Euler integral formula \eqref{eq:2F1Int}. Note that we could
  alternatively proceed as in the proof of Corollary \ref{co:riesz} via the
  Newton binomial series \eqref{eq:binom}.
\end{proof}

\begin{lemma}[$d=2,s=-1$]\label{le:varphi:d2s-1} If $(d,s)=(2,-1)$, then%
  \begin{align*}
    \varphi(\lambda)
    &=
    -\frac{1}{4\gamma}\int_0^1(\lambda+r)E\left(\frac{4\lambda r}{(\lambda+r)^2}\right)
    \frac{r}{\sqrt{1-r^2}}\,\mathrm{d}r
      +\frac{\pi^2}{64\gamma}\lambda^2,\quad\lambda\geq0,\\
    \varphi'(\lambda)
    &=-\frac{1}{8\gamma}\int_0^1
    \left[
      \Bigr(1+\frac{r}{\lambda}\Bigr)
      E\left(\frac{4\lambda r}{(r+\lambda)^2}\right)
      +
      \Bigr(1-\frac{r}{\lambda}\Bigr)
      K\left(\frac{4\lambda r}{(r+\lambda)^2}\right)
    \right]
    \frac{r}{\sqrt{1-r^2}}\,\mathrm{d}r +
      \frac{\pi^2}{32\gamma}\lambda,\quad\lambda>0,\\
    \varphi(\lambda)
    &=
      -\frac{\pi}{8\gamma}\lambda \,{}_2F_1%
      \Bigr( -\frac{1}{2},-\frac{1}{2};\frac{3}{2};
      \frac{1}{\lambda^2}\Bigr)+\frac{\pi^2}{64\gamma}\lambda^2,\quad\lambda\geq1\\
    \varphi'(\lambda)
    &=-\frac{\pi}{16\gamma}\left(\sqrt{1-\frac{1}{\lambda^2}}-\frac{\arcsin\left(\frac{1}{\lambda
      }\right)}{\frac{1}{\lambda}}\right) +\frac{\pi^2}{32\gamma}\lambda,\quad\lambda>1.
  \end{align*}
\end{lemma}

\begin{proof}
  By combining Lemmas \ref{le:varphi} and \ref{le:landend} with $d=2$, we get,
  for $\lambda\geq0$,
  \begin{equation}\label{eq:vphid=2s-1}
    \varphi(\lambda)
    =
    -\frac{R}{2\pi}\int_0^1(\lambda+r)
    E\left(\frac{4\lambda r}{(\lambda+r)^2}\right)
    \frac{r}{\sqrt{1-r^2}}\,\mathrm{d}r
    +\gamma R^2\lambda^2.
  \end{equation}
  Next, by using the well-known ordinary differential equations (for $0<z<1$)
  \begin{equation}
    K'(z)=\frac{E(z)-(1-z)K(z)}{2(1-z)z}%
    \quad\text{and}\quad%
    E'(z)=\frac{E(z)-K(z)}{2z}
  \end{equation}
  we get, after some algebra,
  \begin{equation}
    \varphi'(\lambda)
    =
    -\frac{R}{\pi}\int_0^1
    \left[
      \Bigr(1+\frac{r}{\lambda}\Bigr)
      E\left(\frac{4\lambda r}{(r+\lambda)^2}\right)
      +
      \Bigr(1-\frac{r}{\lambda}\Bigr)
      K\left(\frac{4\lambda r}{(r+\lambda)^2}\right)
    \right]
    \frac{r}{\sqrt{1-r^2}}\,\mathrm{d}r
    +2\gamma R^2\lambda.
  \end{equation}
  By combining \eqref{eq:EllipticE} with the quadratic transformation
  \eqref{eq:quad2F14E} we get, for $z\in[0,1)$,
  \begin{equation}
    (1+z)E\Bigr(\frac{4z}{(1+z)^2}\Bigr)
    =(1+z)\frac{\pi}{2}\,{}_2F_1\Bigr(-\frac{1}{2},\frac{1}{2};1;\frac{4z}{(1+z)^2}\Bigr)
    =\frac{\pi}{2}\,{}_2F_1\Bigr(-\frac{1}{2},-\frac{1}{2};1;z^2\Bigr).
    \label{eq:IvoryE}
  \end{equation}
  Thus \eqref{eq:vphid=2s-1} gives, with $z=\frac{r}{\lambda}\in[0,1)$,
  $r\in[0,1]$, and $\lambda>1$,
  \begin{align*}
    \frac{2}{\pi}\int_0^1
    \Bigr(1+\frac{r}{\lambda}\Bigr)
    E\Bigr(\frac{4\lambda r}{(\lambda+r)^2}\Bigr)
    \frac{r}{\sqrt{1-r^2}}\mathrm{d}r
    &=
      \frac{2}{\pi}
      \int_0^1(1+z)E\Bigr(\frac{4z}{(1+z)^2}\Bigr)\frac{r}{\sqrt{1-r^2}}\mathrm{d}r\\
    &=
      \int_0^1
      \,{}_2F_1\Bigr(-\frac{1}{2},-\frac{1}{2};1;z^2\Bigr)\frac{r}{\sqrt{1-r^2}}\mathrm{d}r\\
    &=\sum_{n=0}^\infty\frac{(-\frac{1}{2})_n^2}{n!^2}
      \Bigr(\frac{1}{\lambda}\Bigr)^{2n}
      \int_0^1\frac{r^{2n+1}}{\sqrt{1-r^2}}\mathrm{d}r\\
    &=\sum_{n=0}^\infty\frac{(-\frac{1}{2})_n^2}{(\frac{3}{2})_nn!}
      \Bigr(\frac{1}{\lambda}\Bigr)^{2n}
    =\,{}_2F_1\Bigr(
    -\frac{1}{2},-\frac{1}{2};\frac{3}{2};
    \Bigr(\frac{1}{\lambda}\Bigr)^2
    \Bigr),
  \end{align*}
  and therefore, using \eqref{eq:vphid=2s-1}, we get, when $\lambda>1$,
  \begin{equation}
    \varphi(\lambda) =
    -R\lambda \,{}_2F_1%
    \Bigr( -\frac{1}{2},-\frac{1}{2};\frac{3}{2};
    \frac{1}{\lambda^2}\Bigr)+\gamma R^2\lambda^2,
  \end{equation}
  hence
  \begin{align}
    \varphi'(\lambda)
    &=-R\,{}_2F_1\Bigr(-\frac{1}{2},-\frac{1}{2};\frac{3}{2};\frac{1}{\lambda^2}\Bigr)
      +\frac{1}{3}\frac{R}{\lambda^2}\,{}_2F_1\Bigr(\frac{1}{2},\frac{1}{2};\frac{5}{2};\frac{R^2}{\lambda
      ^2}\Bigr)+2\gamma R^2\lambda\nonumber\\
    &=-\frac{R}{2}\left(\sqrt{1-\frac{1}{\lambda^2}}-\frac{\arcsin\left(\frac{1}{\lambda
      }\right)}{\frac{1}{\lambda}}\right) +2\gamma R^2\lambda.\label{eq:d=2s=-1phip3}
  \end{align}
  Finally, it remains to recall that $R=\frac{\pi}{8\gamma}$.
\end{proof}

\begin{lemma}[$d=3$, $s=0$]\label{le:varphi:d3s0}
  If $(d,s)=(3,0)$ then, for $\lambda\geq0$,
  \begin{align*}
    \varphi(\lambda)
    & = \frac{1+\log(3\gamma)}{2}-\frac{1}{2\pi\lambda}
    \int_0^1\frac{\bigr((\lambda +r)^2 \log ((\lambda +r)^2)-(\lambda -r)^2
    \log \left((\lambda -r )^2\right)\bigr)r}{\sqrt{1-r^2}}\,\mathrm{d}r
    + \frac{\lambda^2}{3}\\
  \intertext{Moreover, if $\lambda\geq1$,}
    \varphi(\lambda)
    &=-\frac{1}{2}+\frac{\log(3\gamma)}{2}+\frac{\lambda^2+2}{3\lambda}\sqrt{\lambda^2-1}-\log\frac{\lambda+\sqrt{\lambda^2-1}}{2}.
  \end{align*}
\end{lemma}

\begin{proof}
  From Lemma \ref{le:varphi} and with the formulas ($b\in(-a,a)$)
  \[
    \int_{-1}^1\log(a-bt)\mathrm{d} t%
    =\frac{(a+b)\log(a+b)-(a-b)\log (a-b)}{b}-2%
    \quad\text{and}\quad %
    \int_0^1\frac{r^2}{\sqrt{1-r^2}}\,\mathrm{d}r=\frac{\pi}{4}
  \]
  with $a=\lambda^2+r^2$ and $b=2r\lambda$, we obtain, for $\lambda\geq0$,
  \[
    \varphi(\lambda)
    =\frac{1}{2}-\log(R)-\int_0^1\frac{(\lambda +r)^2
      \log ((\lambda +r)^2)-(\lambda -r)^2 \log \left((\lambda -r
        )^2\right)}{2\pi\lambda}\frac{r}{\sqrt{1-r^2}}\,\mathrm{d}r +\gamma R^2\lambda^2.
  \]
  Recall that $R=\frac{1}{\gamma\sqrt{3}}$. It follows that if $\lambda>1$,
  \begin{align*}
    \varphi(\lambda)-\frac{\lambda^2}{3}
    &=\frac{1}{2}+\frac{\log(3\gamma)}{2}-\frac{\lambda}{\pi}\int_0^1\frac{\log(\lambda +r)-\log(\lambda-r)}{\sqrt{1-r^2}}r\mathrm{d}r\\
    &\quad -\frac{2}{\pi}\int_0^1\frac{\log(\lambda +r)+\log(\lambda -r)}{\sqrt{1-r^2}}r^2\mathrm{d}r\\
    &\quad -\frac{1}{\pi\lambda}\int_0^1\frac{\log(\lambda +r)-\log(\lambda -r)}{\sqrt{1-r^2}}r^3\mathrm{d}r.
  \end{align*}
  These three last integrals can be explicitly computed and we obtain the
  desired formula.
\end{proof}

\section{Proof of Riesz formula due to Riesz}
\label{ap:le:riesz}
\label{appC}

\subsection{Cross-ratio}
\label{appC.1}

Recall that in projective geometry, the \emph{cross-ratio} (\emph{birapport}
in French) of four distinct points $z_1,z_2,z_3,z_4$ on the Riemann sphere
$\mathbb{C}\cup\{\infty\}$ is defined by
\[
  [z_1,z_2;z_3,z_4]
  =\frac{z_3-z_1}{z_3-z_2}/\frac{z_4-z_1}{z_4-z_2}
  =\frac{(z_3-z_1)(z_4-z_2)}{(z_3-z_2)(z_4-z_1)},
\]
where each length is removed from the formula if it involves the point at
infinity. The following lemma is a classical and important result of
projective geometry.

\begin{lemma}[Cross-ratio invariance]\label{le:crossratio}
  The cross-ratio is invariant under the Möbius
  transform
  \[
    z\mapsto \frac{az+b}{cz+d},\quad ad-bc\neq0,
  \]
  and thus its modulus is invariant under the ``conjugated Möbius transform''
  $\displaystyle z\mapsto\frac{a\bar{z}+b}{c\bar{z}+d}$, $ad-bc\neq0$.
\end{lemma}

\subsection{Inversions}
\label{appC.2}

In $\mathbb{R}^d$, $d\geq1$, the inversion with center $x_0$ and radius $R>0$
is the transform that maps $x\neq x_0$ to $T(x)$ on the half line started from
$x_0$ and passing through $x$, in such a way that
\[
  \left|x-x_0\right|\; \left|T(x)-x_0\right| = R^2.
\]
The circle centered at $x_0$ and of radius $R$ is pointwise invariant under
the transformation in the sense that all its elements are fixed points of the
transformation. The transformation maps the interior of this circle to its
exterior, and vice versa. In projective geometry, this transformation is
extended to the $d$-dimensional sphere by mapping $x_0$ to the point at
infinity $\infty$, and vice versa. We have
\[
  T(x)-x_0=\frac{R^2}{|x-x_0|^2}(x-x_0),
\]
which exchanges $x_0$ and $\infty$. In dimension $d=2$, using complex numbers,
$T(z)-z_0 = R^2 / (\overline{z-z_0})$, which is a special case of the
conjugated Möbius transform
$z\mapsto\frac{\alpha\bar{z}+\beta}{\gamma\bar{z}+\delta}$. It is worth
mentioning that inversions are geometric transformations at the basis of the
Kelvin transform of functions $\mathbb{R}^d\to\mathbb{R}$.

\begin{lemma}[Classical properties of inversions]\label{le:inversion}
  Let $T$ be the inversion of $\mathbb{R}^d$, $d\geq1$, with center
  $x_0\in\mathbb{R}^d$ and radius $R>0$. Then we have the following
  properties.
  \begin{enumerate}
  \item For all $x$,
    $\displaystyle\left|x-T(x)\right|=\frac{|R^2-\left|x-x_0\right|^2|}{\left|x-x_0\right|}$.
  \item For all $x,y$,
    $\displaystyle\left|T(x)-T(y)\right|=R^2\frac{\left|x-y\right|}{\left|x-x_0\right|\left|y-x_0\right|}$.
  \item As differential forms $\displaystyle\frac{\mathrm{d} T(x)}{|T(x)-x_0|^d}=\frac{\mathrm{d} x}{|x-x_0|^d}$.
  \item The modulus of the cross-ratio of distinct coplanar points
    $x_1,x_2,x_3,x_4$ is invariant under $T$.
  \end{enumerate}
\end{lemma}

\begin{proof} We can assume without loss of generality that $x_0=0$.
  \begin{enumerate}
  \item Since $0,x,T(x)$ are aligned with $0$ at the edge we have
    \[
      |x-T(x)|
      =||x|-|T(x)||
      =\left||x|-\frac{R^2}{|x|}\right|
      =\frac{||x|^2-R^2|}{|x|}.
    \]
  \item We have
    \begin{align*}
      |T(x)-T(y)|^2
      &=|T(x)|^2+|T(y)|^2-2\langle T(x),T(y)\rangle\\
      &=\frac{R^4}{|x|^2}+\frac{R^4}{|y|^2}-2\frac{R^4}{|x|^2|y|^2}\langle x,y\rangle
      =\frac{R^4}{|x|^2|y|^2}|x-y|^2.
    \end{align*}
  \item We have
    $\displaystyle\mathrm{Jac}(T)(x)=\frac{R^2}{|x|^2}(I_d+u\otimes v)$,
    $\displaystyle u=\frac{x}{|x|}$, $\displaystyle v=-2\frac{x}{|x|}$,
    which gives then
    \[
      |\det\mathrm{Jac}(T)(x)|=\left(\frac{R^2}{|x|^2}\right)^d=\left(\frac{|T(x)|}{|x|}\right)^d,
    \]
    via the ``matrix determinant lemma''
    $\det(A+u\otimes v)=(1+u\cdot A^{-1}v)\det(A)$, the determinant analogue
    of the Sherman\,--\,Morrison formula
    $(A+u\otimes v)^{-1}=A^{-1}-\frac{A^{-1}u\otimes v A^{-1}}{1+v\cdot
      A^{-1}u}$.
\item Follows from the fact that $T$ restricted to the plane is a conjugated
  Möbius transform.
\end{enumerate}
\end{proof}

\subsection{Intersecting chords}
\label{appC.3}

The \emph{intersecting chords theorem} in Euclidean (planar) geometry states
that if $AA^*$ and $BB^*$ are two chords of a circle, intersecting at the
point $M$, see Figure \ref{fi:geom}, then
\[
  \mathrm{AM}\times\mathrm{MA^*}=\mathrm{BM}\times\mathrm{MB^*}.
\]
Indeed, the triangles $A^*MB$ and $AMB^*$ are similar, identical up to
rotation and scaling, more precisely they have two equal angles:
$\widehat{A^*MB}=\widehat{AMB^*}$ (opposite angles) and
$\widehat{MA^*B}=\widehat{MB^*A}$ (subtend the same arc).

\begin{figure}[htbp]
  \centering
  \begin{tikzpicture}[scale=.3]
    \draw(0,0) circle (4.01cm);
    \draw (2.89,2.77)-- (-2.83,-2.83);
    \draw (3.96,-0.6)-- (-2.86,2.8);
    \draw [dash pattern=on 5pt off 5pt] (-2.86,2.8)-- (2.89,2.77);
    \draw [dash pattern=on 5pt off 5pt] (3.96,-0.6)-- (-2.83,-2.83);
    \fill [color=black] (2.89,2.77) circle (2pt);
    \draw[color=blue,above right] (3.08,3.04) node {$A^*$};
    \fill [color=black] (-2.83,-2.83) circle (2pt);
    \draw[color=blue,below left] (-2.68,-2.58) node {$A$};
    \fill [color=black] (3.96,-0.6) circle (2pt);
    \draw[color=blue,right] (4.14,-0.34) node {$B^*$};
    \fill [color=black] (-2.86,2.8) circle (1.5pt);
    \draw[color=blue,left] (-2.7,3.06) node {$B$};
    \fill [color=black] (0.97,0.89) circle (3pt);
    \draw[color=black,right] (1.14,1.15) node {$M$};
  \end{tikzpicture}
  \begin{tikzpicture}[scale=.3]
    \draw(0,0) circle (4.01cm);
    \draw (-1,3.87)-- (-2.83,-2.83);
    \draw (4.01,0)-- (-4.01,0);
    \draw [dash pattern=on 5pt off 5pt] (-4.01,0)-- (-1,3.87);
    \draw [dash pattern=on 5pt off 5pt] (4.01,0)-- (-2.83,-2.83);
    \fill [color=black] (-1,3.87) circle (2pt);
    \draw[color=blue,above] (-1,3.87) node {$A^*$};
    \fill [color=black] (-2.83,-2.83) circle (2pt);
    \draw[color=blue,below left] (-2.68,-2.58) node {$A$};
    \fill [color=black] (4.01,0) circle (2pt);
    \draw[color=blue,right] (4.01,0) node {$B^*$};
    \fill [color=black] (-4.01,0) circle (1.5pt);
    \draw[color=blue,left] (-4.01,0) node {$B$};
    \fill [color=black] (-2.05,0) circle (3pt);
    \draw[color=black,above left] (-1.6,-0.1) node {$M$};
    \fill [color=black] (0,0) circle (4pt);
    \draw[color=black,above] (0,0) node {$O$};
  \end{tikzpicture}
  \begin{tikzpicture}[scale=.3]
    \draw(0,0) circle (4.01cm);
    \draw (-1,3.87)-- (-2.83,-2.83);
    \draw (4.01,0)-- (-4.01,0);
    \draw [dash pattern=on 5pt off 5pt] (-4.01,0)-- (-1,3.87);
    \draw [dash pattern=on 5pt off 5pt] (4.01,0)-- (-2.83,-2.83);
    \fill [color=black] (-1,3.87) circle (2pt);
    \draw[color=blue,above] (-1,3.87) node {$x^*$};
    \fill [color=black] (-2.83,-2.83) circle (2pt);
    \draw[color=blue,below left] (-2.68,-2.58) node {$x$};
    \fill [color=black] (4.01,0) circle (2pt);
    \draw[color=blue,right] (4.01,0) node {$z^*$};
    \fill [color=black] (-4.01,0) circle (1.5pt);
    \draw[color=blue,left] (-4.01,0) node {$z$};
    \fill [color=black] (-2.05,0) circle (3pt);
    \draw[color=black,above left] (-1.7,-0.2) node {$y$};
    \fill [color=black] (0,0) circle (4pt);
    \draw[color=black,above] (0,0) node {$O$};
  \end{tikzpicture}
  \caption{\label{fi:geom}Intersecting chords of a circle, $AA^*$ and $BB^*$
    in the first two pictures, $xx^*$ and $zz^*$ for the third. On the two
    last pictures, the chords $BB^*$ and $zz^*$ are diameters of the circle.
    On the right, $x,y\in\mathbb{R}^d$, $d\geq2$, $|x|=r$, $|y|<r$, $x^*$ is
    aligned with $x$ and $y$, $y$ separates $x$ and $x^*$.}
\end{figure}
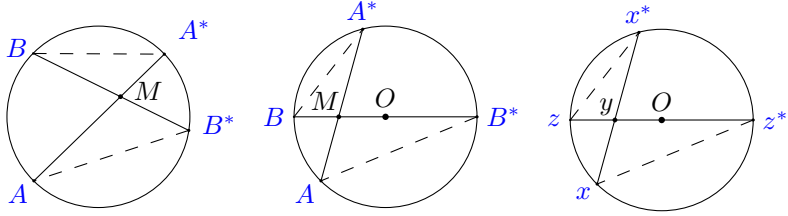

Suppose now that the circle has center $O$, radius $r$, that $BB^*$ is a
diameter, and that $M$ belongs to the segment $OB$ (instead $OB^*$).
Then $\mathrm{BM}=r-\mathrm{OM}$ while $\mathrm{MB^*}=\mathrm{OM} + r$ and
thus
\[
  \text{BM}\times\mathrm{MB^*}=(r-\text{OM})(\mathrm{OM}+r)=r^2-\mathrm{OM}^2.
\]
In Euclidean geometry, this quantity is known as the Laguerre power of the
point $M$ with respect to the circle. We deduce immediately the following
lemma.

\begin{lemma}[Intersecting chords]\label{le:interchords}
  For every chord $AA^*$ of a circle with center $O$ and radius $r$,
  intersecting an arbitrary diameter at point $M$, see Figure \ref{fi:geom},
  we have
  \[
    \mathrm{AM}\times\mathrm{MA^*} = r^2-\mathrm{OM}^2.
  \]
\end{lemma}

\subsection{Riesz geometric argument}
\label{appC.4}

The argument is essentially two-dimensional and involves projective geometry.
Fix $r>0$ and $x,y\in\mathbb{R}^d$, $d\geq2$, with $|y|<r$. Let us define the
map $S:x\mapsto S(x)=x^*$ where $x^*\in\mathbb{R}^d$ is the point aligned with
$x, y$ such that $y$ separates $x$ and $x^*$ and
\[
  |x-y| \; |y-x^*|=r^2-|y|^2.
\]
The map $S$ is the composition of an inversion centered at $y$ of radius
$\sqrt{r^2-|y|^2}$ and the central symmetry centered at $y$ (recall that $y$
separates $x$ and $x^*$). Moreover, by Lemma \ref{le:interchords}, see also
Figure \ref{fi:geom}, we have $|x|=r$ if and only if $|x^*|=r$, namely thea
centered sphere of radius $r$ is globally invariant under $S$. The points $y$
and $\infty$ are mapped to each other by $S$.

Let $T$ be the inversion centered at the origin and with radius $r$. By Lemma
\ref{le:inversion}, the modulus of the cross-ratio of the coplanar points
$x,T(y),y,T(x)$ satisfies
\[
  |[x,T(y);y,T(x)]|
  =\frac{\left|x-y\right|\left|T(x)-T(y)\right|}{\left|x-T(x)\right|\left|y-T(y)\right|}
  =\frac{|x-y|^2r^2}{\left|r^2-|x|^2\right|\left|r^2-|y|^2\right|}.
\]
Note that since $x,y,x^*$ are aligned, the points $x,y,x^*,T(x),T(y)$ are
coplanar.

\begin{lemma}[Commutation]\label{le:commute}
  $S$ and $T$ commute.
\end{lemma}

This is related to the fact that $S$ leaves globally invariant the fixed
points (circle) of $T$.

\begin{proof}
  Using complex coordinates $T(z)=r^2/\overline{z}$ while
  $T(z)-z_0=-(r^2-|z_0|^2)/(\overline{z-z_0})                                             $, where $z_0$ stands for $y$. Now
  we have
  \[
    T(S(z))
    =\frac{r^2}{\overline{z_0}-\frac{r^2-|z_0|^2}{z-z_0}}
    =\frac{r^2(z-z_0)}{\overline{z_0}z-r^2}
  \quad\text{and}\quad
    S(T(z))
    =z_0-\frac{r^2-|z_0|^2}{\overline{\frac{r^2}{\overline{z}}-z_0}}
    =\frac{r^2(z_0-z)}{r^2-\overline{z_0}z}.
  \]
\end{proof}

Since $S$ is the composition of an inversion and a central symmetry, it is a
special case of a conjugate Möbius transform, and then, by Lemma
\ref{le:crossratio},
\(
  |[x,T(y);y,T(x)]|
  =
  |[S(x),S(T(y));S(y),S(T(x))]|.
\)
Since $S$ and $T$ commute (Lemma \ref{le:commute}), we have, using
Lemma \ref{le:inversion} for the final step,
\begin{align*}
  |[x,T(y);y,T(x)]|
  &=|[S(x),T(S(y));S(y),T(S(x))]|
  =|[x^*,T(\infty);\infty,T(x^*)]|
  =|[x^*,0;\infty,T(x^*)]|\\
  &=\frac{|T(x^*)|}{|T(x^*)-x^*|}
  =\frac{|T(x^*)||x^*|}{|r^2-|x^*|^2|}
  =\frac{r^2}{|r^2-|x^*|^2|}.
\end{align*}
It follows that in the case $|x|<r$ (in other words $|x^*|>r$) we get
(recall that $|y|<r$)
\[
  \frac{|x-y|^2}{(r^2-|x|^2)(r^2-|y|^2)}
  =
  \frac{1}{|x^*|^2-r^2}
  \quad\text{hence}\quad
  \frac{1}{(r^2-|x|^2)^{\frac{\alpha}{2}}|x-y|^{-\alpha}}
  =
  \frac{(r^2-|y|^2)^{\frac{\alpha}{2}}}{(|x^*|^2-r^2)^{\frac{\alpha}{2}}}.
\]
Finally, using this formula, we get, for all $y\in\mathbb{R}^d$, $|y|\leq r$,
and all $\alpha\geq0$, $d\geq2$,
\[
  I(y)
  :=\int_{|x|\leq r}
  \frac{\mathrm{d} x}
  {(r^2-|x|^2)^{\frac{\alpha}{2}}|x-y|^{d-\alpha}}
  =(r^2-|y|^2)^{\frac{\alpha}{2}}\int_{|x^*|\geq r}
  \frac{\mathrm{d} x^*}
  {(|x^*|^2-r^2)^{\frac{\alpha}{2}}|x^*-y|^d},
\]
where the differential identity
$\frac{\mathrm{d} x}{|x-y|^d}=\frac{\mathrm{d} x^*}{|x^*-y|^d}$ comes from Lemma
\ref{le:inversion} applied to $S$ which is not an inversion but which is the
composition of an inversion with an isometry (central symmetry).

Using spherical coordinates with $\rho=|x_*|$ and the Funk\,--\,Hecke formula
\eqref{eq:funkhecke} we get
\begin{align}
  I(y)
  &=(r^2-|y|^2)^{\frac{\alpha}{2}}\int_{|x^*|\geq r}
  \frac{\mathrm{d} x^*}
  {(|x^*|^2-r^2)^{\frac{\alpha}{2}}(|x^*|^2-2x^*\cdot y+|y|^2)^{\frac{d}{2}}}\nonumber\\
  &=|S_1|
    \frac{\Gamma(\frac{d}{2})}{\sqrt{\pi}\Gamma(\frac{d-1}{2})}
    (r^2-|y|^2)^{\frac{\alpha}{2}}\int_r^\infty\int_0^\pi
    \frac{\rho^{d-1}\sin^{d-2}(\theta)\mathrm{d} \rho\mathrm{d} \theta}
    {(\rho^2-r^2)^{\frac{\alpha}{2}}(\rho^2-2\rho|y|\cos(\theta)+|y|^2)^{\frac{d}{2}}}\nonumber\\
  &=|S_1|\frac{\Gamma(\frac{d}{2})}{\sqrt{\pi}\Gamma(\frac{d-1}{2})}
    (r_1^2-1)^{\frac{\alpha}{2}}
    \int_{r_1}^\infty\frac{\rho_1^{d-1}}{(\rho_1^2-r_1^2)^{\frac{\alpha}{2}}}
    \Bigr(\int_0^\pi\frac{\sin^{d-2}(\theta)\mathrm{d} \theta}
    {(\rho_1^2-2\rho_1\cos(\theta)+1)^{\frac{d}{2}}}\Bigr)\mathrm{d} \rho_1
    \label{eq:I(y)}
\end{align}
where $r:=r_1|y|$ and $\rho:=\rho_1|y|$. Note
that $r_1\geq1$ and $\rho_1\geq1$.

\subsection{Trigonometric change of variable}\label{se:trigo}
\label{appC.5}

Let us show that for $d>1$ and $\rho_1>1$,
\begin{equation}\label{eq:id}
  i_d:=\int_0^\pi
  \frac{\sin ^{d-2}(\theta )}{\left(\rho_1 ^2-2 \rho_1  \cos (\theta )+1\right)^{\frac{d}{2}}}
  \mathrm{d} \theta = \frac{\rho_1 ^{2-d}}{\rho_1 ^2-1}
  \int_0^\pi\sin ^{d-2}(\alpha)\mathrm{d} \alpha
  =\frac{\rho_1 ^{2-d}}{\rho_1 ^2-1}\sqrt{\pi }
  \frac{\Gamma \left(\frac{d-1}{2}\right)}{\Gamma \left(\frac{d}{2}\right)}.
\end{equation}
We first give a historical geometric argument. We then give in Remark
\ref{rm:vignat} an analytic argument using properties of the Gegenbauer
polynomials. The second equality in \eqref{eq:id} follows from the fact that
the middle integral becomes an Euler beta integral after the change of
variable $u=\sin(\alpha)$. To prove the first equality in \eqref{eq:id}, we
follow \cite[p.~400]{MR0350027}, and we use the change of variable
\[
  \frac{\sin(\theta)}{\sqrt{\rho_1^2-2\rho_1\cos(\theta)+1}}
  =\frac{\sin(\alpha)}{\rho_1},
\]
see Figure \ref{fi:geomid} for a geometric interpretation\footnote{It is
  mentioned in \cite[p.~400]{MR0350027} that this change of variable was
  suggested S.I.~Greenberg. Nevertheless such geometric reasoning goes back at
  least to the works on elliptic integrals of the 19-th century, see
  \cite{MR0099454}.}.
Following this figure, we have the identity
\[
  \rho_1^2-2\rho_1\cos(\theta)+1
  =A(\alpha)^2
  \quad\text{where}\quad
  A(\alpha)=\sqrt{\rho_1^2-\sin^2(\alpha)}+\cos(\alpha),
\]
hence
\(
  2\rho_1\sin(\theta)\mathrm{d} \theta
  =2A(\alpha)A'(\alpha)\mathrm{d} \alpha
\)
and by using the formula for the change of variable this gives
\[
  \mathrm{d} \theta
  =\frac{A'(\alpha)}{\sin(\alpha)}\mathrm{d} \alpha
  =\frac{-\sin(\alpha)-\frac{\sin(\alpha)\cos(\alpha)}{\sqrt{\rho_1^2-\sin^2(\alpha)}}}{\sin(\alpha)}\mathrm{d} \alpha
  =-\left(\frac{\sqrt{\rho_1^2-\sin^2(\alpha)}+\cos(\alpha)}{\sqrt{\rho_1^2-\sin^2(\alpha)}}\right)
    \mathrm{d} \alpha.
\]
Therefore,
we obtain, noting that $\theta = 0 \Longleftrightarrow \alpha = \pi$ and
$\theta = \pi \Longleftrightarrow \alpha = 0$ (see Figure~\ref{fi:geomid}),
\begin{align*}
  i_d
  & =  \int_0^\pi
  \Bigr(\frac{\sin(\alpha)}{\rho_1}\Bigr)^{d-2}
  \frac{1}{\left(\cos(\alpha)+\sqrt{\rho_1^2-\sin^2(\alpha)}\right)^2}
  \frac{\sqrt{\rho_1^2-\sin^2(\alpha)}+\cos(\alpha)}{\sqrt{\rho_1^2-\sin^2(\alpha)}}
    \mathrm{d} \alpha\\
  &=  \int_0^\pi
  \Bigr(\frac{\sin(\alpha)}{\rho_1}\Bigr)^{d-2}
    \frac{1}{\cos(\alpha)+\sqrt{\rho_1^2-\sin^2(\alpha)}}
    \frac{1}{\sqrt{\rho_1^2-\sin^2(\alpha)}}
    \mathrm{d} \alpha\\
  & =  \int_0^\pi
  \Bigr(\frac{\sin(\alpha)}{\rho_1}\Bigr)^{d-2}
    \frac{\cos(\alpha)-\sqrt{\rho_1^2-\sin^2(\alpha)}}{\cos^2(\alpha)-(\rho_1^2-\sin^2(\alpha))}
    \frac{1}{\sqrt{\rho_1^2-\sin^2(\alpha)}}
    \mathrm{d} \alpha\\
  & =  \frac{1}{\rho_1^{d-2}(1-\rho_1^2)}
    \int_0^\pi(\sin(\alpha))^{d-2}
    \left(\frac{\cos(\alpha)}{\sqrt{\rho_1^2-\sin^2(\alpha)}}-1\right)
    \mathrm{d} \alpha\\
  &=\frac{1}{\rho_1^{d-2}(\rho_1^2-1)}
    \int_0^\pi(\sin(\alpha))^{d-2}
   \mathrm{d} \alpha,
\end{align*}
where the last equality follows from the antisymmetry of $\cos$ around $\pi/2$.
This proves \eqref{eq:id}.

\begin{remark}[Proof of \eqref{eq:id} using Gegenbauer polynomials]\label{rm:vignat}
  Let $\rho = \frac{1}{\rho_1} < 1$. Using the generating function for
  Gegenbauer polynomials
  $(1 - 2 \rho \cos \theta + \rho^2)^{-\frac{d}{2}} = \sum_{n = 0}^\infty
  C_n^{(\frac{d}{2})}(\cos \theta) \; \rho^n$ gives
  \begin{equation}\label{E:Ggenfun}
    \int_0^\pi \frac{\sin^{d-2}\theta}{(1 - 2 \rho \cos \theta + \rho^2)^\frac{d}{2}} \; \mathrm{d} \theta
    = \sum_{n = 0}^\infty \rho^n
    \int_0^\pi \sin^{d-2}\theta \;C_n^{(\frac{d}{2})}(\cos \theta) \, \mathrm{d} \theta.
  \end{equation}
  The integral on the right-hand side vanishes for odd degree $n$ since the
  Gegenbauer polynomials are odd functions of $\cos \theta$.
  For even degree $n = 2k$, the integral can be computed as
  \begin{equation}\label{E:Gint}
    \int_0^\pi \sin^{d-2}\theta \;C_{2k}^{(\frac{d}{2})}(\cos \theta) \, \mathrm{d} \theta =
    \int_0^\pi \sin^{d-2} \theta \; \mathrm{d} \theta =
    \sqrt{\pi} \frac{\Gamma\left(\frac{d-1}{2}\right)}{\Gamma\left(\frac{d}{2}\right)}.
  \end{equation}
  To establish (\ref{E:Gint}),
  use the recurrence relation~\cite[(18.9.7)]{NIST:DLMF}
  \[
    C_{2k}^{(\frac{d}{2})}(x) = C_{2k-2}^{(\frac{d}{2})}(x) +
    \frac{2k + \frac{d}{2} -1}{\frac{d}{2}-1} C_{2k}^{(\frac{d}{2}-1)}(x)
  \]
  and integrate against $\sin^{d-2}\theta$ to produce
  \[
    \int_0^\pi \sin^{d-2}\theta \;C_{2k}^{(\frac{d}{2})}(\cos \theta) \, \mathrm{d} \theta =
    \int_0^\pi \sin^{d-2}\theta \;C_{2k-2}^{(\frac{d}{2})}(\cos \theta) \, \mathrm{d} \theta +
    \frac{2k + \frac{d}{2} -1}{\frac{d}{2}-1} \int_0^\pi \sin^{d-2}\theta \; C_{2k}^{(\frac{d}{2}-1)}(\cos \theta).
  \]
  The second integral on the right-hand side vanishes by orthogonality, so
  \[
    \int_0^\pi \sin^{d-2}\theta \;C_{2k}^{(\frac{d}{2})}(\cos \theta) \, \mathrm{d} \theta =
    \int_0^\pi \sin^{d-2}\theta \;C_{2k-2}^{(\frac{d}{2})}(\cos \theta) \, \mathrm{d} \theta,
  \]
  with repeated application giving the first equality in (\ref{E:Gint}).
  Using (\ref{E:Gint}) in (\ref{E:Ggenfun}) then gives
  \[
    \int_0^\pi \frac{\sin^{d-2}\theta}{(1 - 2 \rho \cos \theta + \rho^2)^\frac{d}{2}} \; \mathrm{d} \theta =
    \sqrt{\pi} \frac{\Gamma\left(\frac{d-1}{2}\right)}{\Gamma\left(\frac{d}{2}\right)}
    \sum_{k=0}^\infty \rho^{2k} =
    \sqrt{\pi} \frac{\Gamma\left(\frac{d-1}{2}\right)}{\Gamma\left(\frac{d}{2}\right)}
    \frac{1}{1 - \rho^2}.
  \]
  The substitution $\rho = \frac{1}{\rho_1}$ then gives \eqref{eq:id}.
\end{remark}

\begin{figure}[htbp]
  \centering
  \begin{tikzpicture}[scale=.5]
    \coordinate (A) at (5,0);
    \coordinate (B) at (1.5,3.71);
    \coordinate (C) at (0,0);
    \coordinate (P) at (2.55,2.55);
    \coordinate (Q) at (1.5,0);
    \draw(C) circle (4cm);
    \draw (-4,0)-- (6,0);
    \fill [color=black] (C) circle (2pt);
    \draw[color=black,below left] (C) node {C};
    \draw (0,0)-- (B);
    \fill [color=black] (B) circle (2pt);
    \draw[color=black,above right] (B) node {B};
    \fill [color=black] (A) circle (2pt);
    \draw[color=black,below right] (A) node {A};
    \draw [very thick] (B)-- (A);
    \draw [very thick] (0,0)-- (A);
    \draw [thin](0,0)-- (P);
    \fill [color=black] (P) circle (2pt);
    \draw[color=black,above right] (P) node {P};
    \draw [thin](B)--(Q);
    \fill [color=black] (Q) circle (2pt);
    \draw[below right] (Q) node {Q};
    \pic [draw,"$\theta$",
    angle eccentricity = 1.3, angle radius = .5 cm,
    thick, color = blue] {angle = A--C--B};
    \pic [draw,"$\alpha$",
    angle eccentricity = 1.3, angle radius = .5 cm,
    thick, color=blue] {angle = C--B--A};
    \pic [draw,angle eccentricity = 1, angle radius = .25 cm,
    thick, color = red]{right angle = A--Q--B};
    \pic [draw,angle eccentricity = 1, angle radius = .25 cm,
    thick, color = red] {right angle = C--P--A};
  \end{tikzpicture}
  \caption{\label{fi:geomid} Geometric interpretation of the $\theta$ to
    $\alpha$ change of variables for $i_d$. The angles and distances are
    $ACB=\theta$, $CBA=\alpha$, $CB = 1$ and $CA=\rho_1$. The right-angled
  triangle $ABQ$ has hypotenuse $AB$, thus
    \begin{equation*}
      AB^2 = BQ^2+AQ^2
      =\sin^2(\theta)+(AC-QC)^2
      =\sin^2(\theta)+(\rho_1-\cos(\theta))^2
      =\rho_1^2-2\rho_1\cos(\theta)+1.
      \end{equation*}
    The sine rule then gives
    \[
      \frac{\sin(\alpha)}{\rho_1} = \frac{\sin(\theta)}{\sqrt{\rho_1^2-2\rho_1\cos(\theta)+1}}.
    \]
    On the other hand, we also have
    \[
      \sqrt{\rho_1^2-2\rho_1\cos(\theta)+1}
      =
      AB
      =AP+PB
      =\sqrt{\rho_1^2-\sin^2(\alpha)}+\cos(\alpha).
    \]}
\end{figure}
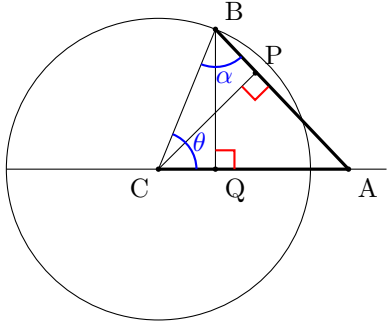

\subsection{Conclusion}
\label{appC.6}

By combining \eqref{eq:I(y)} and \eqref{eq:id}, using
the successive changes of variables $t=\rho_1^2-r_1^2$, $t_1=t/(r_1^2-1)$, and
$u=1/(1+t_1)$, and the Euler reflection formula \eqref{eq:reflection}, we get
\begin{align*}
  I(y)
  &=|S_1|(r_1^2-1)^{\frac{\alpha}{2}}
    \int_{r_1}^\infty\frac{\rho_1\mathrm{d} \rho_1}{(\rho_1^2-r_1^2)^{\frac{\alpha}{2}}(\rho_1
    ^2-1)}\\
  &=|S_1|\frac{(r_1^2-1)^{\frac{\alpha}{2}}}{2}
    \int_0^\infty\frac{\mathrm{d} t}{t^{\frac{\alpha}{2}}(t+r_1 ^2-1)}\\
  &=\frac{|S_1|}{2}
    \int_0^\infty\frac{\mathrm{d} t_1}{t_1^{\frac{\alpha}{2}}(t_1+1)}\\
  &=\frac{|S_1|}{2}
    \int_0^1\frac{u^{\frac{\alpha}{2}-1}\mathrm{d} u}{(1-u)^{\frac{\alpha}{2}}}
  =\frac{|S_1|\Gamma(\frac{\alpha}{2})\Gamma(1-\frac{\alpha}{2})}{2}
  =\frac{\pi^{\frac{d}{2}+1}}{\Gamma(\frac{d}{2})\sin(\frac{\pi}{2}\alpha)}.
\end{align*}
This completes the proof of \eqref{eq:riesz} and thus of Lemma \ref{le:riesz}.

\section{Alternative analytic proof of Riesz formula}
\label{ap:se:analytic:proof:riesz}

\subsection{The Mellin transform and Riesz potentials}


Quoting \cite[Ch.~12]{MR1867914}, we recall that the Fourier transform pair may be written in the form
\begin{equation*}
    A(\theta):=\int_{\mathbb{R}}a(t)\mathrm{e}^{\mathrm{i}\theta t}\mathrm{d}t,\quad\alpha<\Im\theta<\beta,
    \quad\text{and}\quad
    a(t)=\frac{1}{2\pi}\int_{\mathrm{i}c+\mathbb{R}}A(\theta)\mathrm{e}^{-\mathrm{i}\theta t}\mathrm{d}\theta,\quad\alpha<c<\beta.
\end{equation*}
The Mellin transform and its inverse follow if we introduce the variable changes
\begin{equation*}
    z=\mathrm{i}\theta,\quad x=\mathrm{e}^t,\quad f(x)=a(\log(x)),
\end{equation*}
so that we obtain the reciprocal pair of integral transforms, for $f:(0,+\infty)\to\mathbb{R}$,
\begin{equation}\label{eq:mellin}
    F(z):=\int_0^\infty f(x)x^{z-1}\mathrm{d}x,\quad\alpha<\Re z<\beta,
    \quad\text{and}\quad
    f(x)=\frac{1}{2\pi\mathrm{i}}\int_{c+\mathrm{i}\mathbb{R}}F(z)x^{-z}\mathrm{d}z,\quad\alpha<c<\beta.
\end{equation}
Display \eqref{eq:mellin} exhibits the Mellin transform followed by its inversion formula. The integral defining the transform normally exists only in the strip $\alpha<\Re z<\beta$; therefore the inversion contour must be placed in this strip. For convenience we also denote by $\mathcal{M}f=F$ the Mellin transform of $f$.

The Mellin transform of $x\mapsto\mathrm{e}^{-x}$ is the Euler Gamma  function $\Gamma$. Its poles are $0,-1,-2,-3,\ldots$. The Mellin
transform of $x\mapsto(1-x)_+^{b-1}$ at a point $z$ is\footnote{Recall the
  definition of the Euler Beta function 
  $\mathrm{Beta}(a,b):=\int_0^1t^{a-1}(1-t)^{b-1}\mathrm{d}t=\frac{\Gamma(a)\Gamma(b)}{\Gamma(a+b)}$.}
$\int_0^1x^{z-1}(1-x)^{b-1}\mathrm{d}x=\mathrm{Beta}(z,b)$.


\begin{lemma}[Riesz potential of radial functions]\label{le:dkk}
  Suppose that
  \[
    x\in\mathbb{R}^d\mapsto f(x)=\varphi(|x|^2),
  \]
  where $\varphi:\mathbb{C}\to\mathbb{R}$ is given as the absolutely convergent inverse Mellin
  transform
  \[
    \varphi(r):=\frac{1}{2\pi\mathrm{i}}
    \int_{\lambda+\mathrm{i}\mathbb{R}}
    \mathcal{M}\varphi(z)r^{-z}\mathrm{d}z,
    \quad
    \text{for some }\lambda\in\mathbb{R}.
  \]
  If $0<\alpha<2\lambda<d$, the Riesz potential
  $(\left|\cdot\right|^{-(d-\alpha)}*f)(x)$ is well defined for $x\neq0$, and
  \[
    (\left|\cdot\right|^{-(d-\alpha)}*f)(x)
    =\psi(|x|^2),
  \]
  where
  \[
    \psi(r):=
    \frac{1}{2\pi\mathrm{i}}
    \frac{\pi^{\frac{d}{2}}\Gamma(\frac{\alpha}{2})}{\Gamma(\frac{d-\alpha}{2})}
    \int_{\lambda-\frac{\alpha}{2}+\mathrm{i}\mathbb{R}}
    \frac{\Gamma(z)\Gamma(\frac{d-\alpha}{2}-z)}
    {\Gamma(\frac{\alpha}{2}+z)\Gamma(\frac{d}{2}-z)}
    \mathcal{M}\varphi(z+\tfrac{\alpha}{2})
    r^{-z}\mathrm{d}z.
  \]
  In other words, the Mellin transform of $\psi$ satisfies
  \[
    \mathcal{M}\psi(z)
    =
    \frac{\pi^{\frac{d}{2}}\Gamma(\frac{\alpha}{2})}
    {\Gamma(\frac{d-\alpha}{2})}
    \frac{\Gamma(z)\Gamma(\frac{d-\alpha}{2}-z)}
    {\Gamma(\frac{\alpha}{2}+z)\Gamma(\frac{d}{2}-z)}
    \mathcal{M}\varphi(z+\tfrac{\alpha}{2}).
  \]
\end{lemma}

\begin{proof}
  This is \cite[Prop.~2 with $V\equiv1$ (and $l=0$)]{MR3640641}, see also
  \cite[eq.~(7)]{MR3640641}. The idea is to use the inverse Mellin
  transform of $\varphi$ to reduce the problem, via the Fubini theorem, to the
  computation of the Riesz potential of inverse powers of the norm, which is
  immediate from the semigroup property of the Riesz kernel. Namely, following
  \cite[eq.~(1.1.12)]{MR0350027} or \cite[eq.~(8)~p.~118]{MR0290095}, on $\mathbb{R}^d$, the semigroup
  property for Riesz kernels reads, for all $\alpha,\beta\in\mathbb{C}$ such
  that $\Re\alpha,\Re\beta>0$ and $\Re\alpha+\Re\beta<d$,
  \begin{equation}\label{eq:semigroup}
    \left|\cdot\right|^{-(d-\alpha)}*\left|\cdot\right|^{-(d-\beta)}
    =\frac{c_d(\alpha)c_d(\beta)}{c_d(\alpha+\beta)}
    \left|\cdot\right|^{-(d-(\alpha+\beta))}
    \quad\text{where}\quad
    c_d(z)
    :=\frac{2^z\pi^{\frac{d}{2}}\Gamma(\frac{z}{2})}{\Gamma(\frac{d-z}{2})}.
  \end{equation}
  Now, by the inverse Mellin transform of $\varphi$, the Fubini theorem, and the semigroup property,
  \begin{align*}
    (\left|\cdot\right|^{-(d-\alpha)}*f)(x)
    &=\frac{1}{2\pi\mathrm{i}}
    \int_{\lambda+\mathrm{i}\mathbb{R}}
      (\left|\cdot\right|^{-(d-\alpha)}*\left|\cdot\right|^{-2z})\mathcal{M}\varphi(z)\mathrm{d}z\\
    &=\frac{1}{2\pi\mathrm{i}}
      \int_{\lambda+\mathrm{i}\mathbb{R}}
      \frac{c_d(\alpha)c_d(d-2z)}{c_d(d+\alpha-2z)}
      \mathcal{M}\varphi(z)\left|\cdot\right|^{-(2z-\alpha)}\mathrm{d}z\\
    &=\frac{1}{2\pi\mathrm{i}}
      \int_{\lambda-\frac{\alpha}{2}+\mathrm{i}\mathbb{R}}
      \frac{c_d(\alpha)c_d(d-\alpha-2w)}{c_d(d-2w)}
      \mathcal{M}\varphi(w+\tfrac{\alpha}{2})\left|\cdot\right|^{-2w}\mathrm{d}w.    
  \end{align*}
\end{proof}

\subsection{Analytic proof of Riesz integral formula}

The Riesz integral formula \eqref{eq:riesz} for $R=1$, $x\in B_1$, $x\neq0$, is a special case of \cite[Cor.~4 with $V\equiv1$, $l=0$, $\alpha=s-d$,  $\delta=d$, $\rho = \sigma = -\frac{d-s}{2}$, $0<s<d-2$ (implies $\sigma>-1$), see also eq.~(7)]{MR3640641}. Let us give the proof extracted from there. With $\varphi(r):=(1-r)_+^{\frac{s-d}{2}}$, we have $\mathcal{M}\varphi(z)=\int_0^1r^{z-1}(1-r)^{\frac{s-d}{2}}\mathrm{d}r=\mathrm{Beta}(z,1-\tfrac{d-s}{2})$, and by 
  Lemma \ref{le:dkk} with $\alpha=d-s$,
  \begin{align*}
    \mathcal{M}\psi(z)
    &=\frac{\pi^{\frac{d}{2}}\Gamma(\frac{d-s}{2})\Gamma(1-\frac{d-s}{2})}
    {\Gamma(\frac{s}{2})}
      \frac{\Gamma(z)\Gamma(\frac{s}{2}-z)}{\Gamma(\frac{d}{2}-z)\Gamma(z+1)}.
  \end{align*}
  Now, if the vertical line $\lambda+\mathrm{i}\mathbb{R}$ separates the poles of $z\mapsto\Gamma(z)$ and the poles of $z\mapsto\Gamma(\frac{s}{2}-z)$, then\footnote{The $\Gamma$ function has no zeros, indeed a zero leads via $\Gamma(z)=(z-1)\Gamma(z-1)$ to infinitely many zeros to the left, and then via $\Gamma(z)\Gamma(1-z)=\frac{\pi}{\sin(\pi z)}$ to infinitely many poles to the right, which contradicts the analycity of $\Gamma$ on $\Re z>0$.

  The $\Gamma$ function is meromorphic on the complex plane; its poles are the non-positive integers, and are simple. Moreover, $\mathrm{Residue}_{z=-n}(\Gamma(z)):=\lim_{z\to-n}(z-(-n))\Gamma(z)=\frac{(-1)^n}{n!}$, which follows from $(z+n)\Gamma(z)=\frac{\Gamma(z+n+1)}{z(z+1)\cdots(z+n-1)}$.
}
  \begin{align*}
  \frac{1}{2\pi\mathrm{i}}\int_{\lambda+\mathrm{i}\mathbb{R}}
        \frac{\Gamma(z)\Gamma(\frac{s}{2}-z)}{\Gamma(\frac{d}{2}-z)\Gamma(z+1)}
        x^{-z}\mathrm{d}z
        &=\sum_{k=0}^\infty\mathrm{Residue}_{z=-k}\Bigr(\frac{\Gamma(z)\Gamma(\frac{s}{2}-z)}{\Gamma(\frac{d}{2}-z)\Gamma(z+1)}x^{-z}\Bigr)\\
        &=\sum_{k=0}^\infty\frac{\Gamma(\frac{s}{2}+k)}{\Gamma(\frac{d}{2}+k)\Gamma(-k+1)}x^{k}\mathrm{Residue}_{z=-k}(\Gamma(z))\\
        &=\sum_{k=0}^\infty\frac{\Gamma(\frac{s}{2}+k)}{\Gamma(\frac{d}{2}+k)\Gamma(-k+1)}\frac{(-x)^{k}}{k!} =\frac{\Gamma(\frac{s}{2})}{\Gamma(\frac{d}{2})}.
  \end{align*}
 
\begin{remark}[Meijer G-functions] A key point in the proof above is the computation of the inverse Mellin transform of a certain ratio of products of Gamma functions (Mellin transfrom of $\psi$). This is actually the definition of Meijer G-functions. Following \cite{MR3640641}, if $f(x):=(1-|x|^2)_+^\sigma\ {}_2F_1(a,b;c;1-|x|^2)=\varphi(|x|^2)$ where $\varphi(r)=(1-r)_+^\sigma\ {}_2F_1(a,b;c;r)$, then it is possible, by using the same method as above, to express $\varphi$ as a Meijer G-function, and to deduce that the Riesz potential of $f$ on the unit ball is equal to another Meijer G-function, which reduces to a hypergeometric function in certain cases.
\end{remark}


\bigskip

\textbf{Acknowledgments.} The authors are grateful to Doug Hardin for the
suggestion of using the spherical Laplacian, to Bent Fuglede and Wolfgang
Wendland for helpful comments on the convolution of distributions, to Franck
Wielonsky for historical references, and to Ekatrina Karatsuba for the
references to tables and related literature. The manuscript further benefited
from enriching comments by an anonymous reviewer (see Remark
\ref{rm:outside}), as well as from the analytic argument based on Gegenbauer
polynomials in Remark \ref{rm:vignat} which was contributed by Christophe
Vignat.

\bibliographystyle{abbrv}%
\bibliography{hdep}%

\begin{thebibliography}{10}

\bibitem{MR510197}
L.~V. Ahlfors.
\newblock {\em Complex analysis}.
\newblock International Series in Pure and Applied Mathematics. McGraw-Hill,
  third edition, 1978.
\newblock An introduction to the theory of analytic functions of one complex
  variable.

\bibitem{MR966232}
G.~Almkvist and B.~Berndt.
\newblock Gauss, {L}anden, {R}amanujan, the arithmetic-geometric mean,
  ellipses, {$\pi$}, and the {\it {l}adies diary}.
\newblock {\em Amer. Math. Monthly}, 95(7):585--608, 1988.

\bibitem{zbMATH01231230}
G.~E. {Andrews}, R.~{Askey}, and R.~{Roy}.
\newblock {\em {Special functions.}}, volume~71.
\newblock Cambridge: Cambridge University Press, 1999.

\bibitem{MR0185155}
W.~N. Bailey.
\newblock {\em Generalized hypergeometric series}.
\newblock Cambridge Tracts in Mathematics and Mathematical Physics, No. 32.
  Stechert-Hafner, Inc., New York, 1964.

\bibitem{MR3970999}
S.~V. Borodachov, D.~P. Hardin, and E.~B. Saff.
\newblock {\em Discrete energy on rectifiable sets}.
\newblock Springer Monographs in Mathematics. Springer, New York, 2019.

\bibitem{MR2460394}
Y.~A. Brychkov.
\newblock {\em Handbook of special functions}.
\newblock CRC Press, Boca Raton, FL, 2008.
\newblock Derivatives, integrals, series and other formulas.

\bibitem{MR0277773}
P.~F. Byrd and M.~D. Friedman.
\newblock {\em Handbook of elliptic integrals for engineers and scientists}.
\newblock Die Grundlehren der mathematischen Wissenschaften, Band 67. Springer,
  1971.
\newblock Second edition, revised.

\bibitem{zbMATH06665921}
J.~A. {Carrillo} and Y.~{Huang}.
\newblock {Explicit equilibrium solutions for the aggregation equation with
  power-law potentials}.
\newblock {\em {Kinet. Relat. Models}}, 10(1):171--192, 2017.

\bibitem{MR3262506}
D.~Chafa\"\i, N.~Gozlan, and P.-A. Zitt.
\newblock First-order global asymptotics for confined particles with singular
  pair repulsion.
\newblock {\em Ann. Appl. Probab.}, 24(6):2371--2413, 2014.

\bibitem{hdep}
D.~Chafaï, E.~B. Saff, and R.~S. Womersley.
\newblock Threshold condensation to singular support for a {Riesz} equilibrium
  problem.
\newblock preprint, 2022.

\bibitem{MR1867914}
B.~Davies.
\newblock {\em Integral transforms and their applications}, volume~41 of {\em
  Texts in Applied Mathematics}.
\newblock Springer-Verlag, New York, third edition, 2002.

\bibitem{NIST:DLMF}
{\it NIST Digital Library of Mathematical Functions}.
\newblock \url{http://dlmf.nist.gov/}, Release 1.1.2 of 2021-06-15.
\newblock F.~W.~J. Olver, A.~B. {Olde Daalhuis}, D.~W. Lozier, B.~I. Schneider,
  R.~F. Boisvert, C.~W. Clark, B.~R. Miller, B.~V. Saunders, H.~S. Cohl, and
  M.~A. McClain, eds.

\bibitem{MR3640641}
B.~Dyda, A.~Kuznetsov, and M.~Kwa\'{s}nicki.
\newblock Fractional {L}aplace operator and {M}eijer {G}-function.
\newblock {\em Constr. Approx.}, 45(3):427--448, 2017.

\bibitem{MR0058756}
A.~Erd\'{e}lyi, W.~Magnus, F.~Oberhettinger, and F.~G. Tricomi.
\newblock {\em Higher transcendental functions. {V}ols. {I}, {II}}.
\newblock McGraw-Hill, 1953.
\newblock Based, in part, on notes left by Harry Bateman.

\bibitem{MR2281163}
G.~Gasper and M.~Schlosser.
\newblock Some curious {$q$}-series expansions and beta integral evaluations.
\newblock {\em Ramanujan J.}, 13(1-3):227--240, 2007.

\bibitem{GutlebCarrilloOlver2021Balls}
T.~S. Gutleb, J.~A. Carrillo, and S.~Olver.
\newblock Computation of power law equilibrium measures on balls of arbitrary
  dimension, 2021.
\newblock preprint
  \href{https://arxiv.org/abs/2109.00843v1}{arXiv:2109.00843v1}.

\bibitem{MR0099454}
H.~Hancock.
\newblock {\em Elliptic integrals}.
\newblock Dover Publications, Inc., New York, 1958.

\bibitem{MR1790156}
S.~Helgason.
\newblock {\em Groups and geometric analysis}, volume~83 of {\em Mathematical
  Surveys and Monographs}.
\newblock American Mathematical Society, Providence, RI, 2000.
\newblock Integral geometry, invariant differential operators, and spherical
  functions, Corrected reprint of the 1984 original.

\bibitem{zbMATH06324458}
Y.~{Huang}.
\newblock {Explicit Barenblatt profiles for fractional porous medium
  equations}.
\newblock {\em {Bull. Lond. Math. Soc.}}, 46(4):857--869, 2014.

\bibitem{zbMATH02006424}
C.~{Krattenthaler} and K.~{Srinivasa Rao}.
\newblock {Automatic generation of hypergeometric identities by the beta
  integral method.}
\newblock {\em {J. Comput. Appl. Math.}}, 160(1-2):159--173, 2003.

\bibitem{doi:10.1098/rstl.1771.0037}
J.~Landen.
\newblock {XXXVI}. {A} disquisition concerning certain fluents, which are
  assignable by the arcs of the conic sections; wherein are investigated some
  new and useful theorems for computing such fluents.
\newblock {\em Philosophical Transactions of the Royal Society of London},
  61:298--309, 1771.

\bibitem{doi:10.1098/rstl.1775.0028}
J.~Landen.
\newblock {XXVI}. {A}n investigation of a general theorem for finding the
  length of any arc of any conic hyperbola, by means of two elliptic arcs with
  some other new and useful theorems deduced therefrom.
\newblock {\em Philosophical Transactions of the Royal Society of London},
  65:283--289, 1775.

\bibitem{MR0350027}
N.~S. Landkof.
\newblock {\em Foundations of modern potential theory}.
\newblock Springer, 1972.
\newblock Translated from the Russian by A. P. Doohovskoy, Die Grundlehren der
  mathematischen Wissenschaften 180.

\bibitem{MR1018814}
Y.~Mizuta.
\newblock Continuity properties of {R}iesz potentials and boundary limits of
  {B}eppo {L}evi functions.
\newblock {\em Math. Scand.}, 63(2):238--260, 1988.

\bibitem{MR1211771}
Y.~Mizuta.
\newblock Continuity properties of potentials and {B}eppo-{L}evi-{D}eny
  functions.
\newblock {\em Hiroshima Math. J.}, 23(1):79--153, 1993.

\bibitem{MR1428685}
Y.~Mizuta.
\newblock {\em Potential theory in {E}uclidean spaces}, volume~6 of {\em GAKUTO
  International Series. Mathematical Sciences and Applications}.
\newblock Gakkotosho Co., Ltd., Tokyo, 1996.

\bibitem{MR0199449}
C.~M{\"{u}}ller.
\newblock {\em Spherical harmonics}, volume~17 of {\em Lecture Notes in
  Mathematics}.
\newblock Springer, 1966.

\bibitem{MR1054647}
A.~P. Prudnikov, Y.~A. Brychkov, and O.~I. Marichev.
\newblock {\em Integrals and series. {V}ol. 3}.
\newblock Gordon and Breach, 1990.
\newblock More special functions, Translated from the Russian by G. G. Gould.

\bibitem{polya-szego}
G.~Pólya and G.~Szegő.
\newblock Über die transfiniten {D}urchmesser ({K}apazitätskonstante) von
  ebenen und räumlichen {P}unktmengen.
\newblock {\em Journal für die reine und angewandte {M}athematik},
  (165):4--49, 1931.

\bibitem{zbMATH03003508}
M.~{Riesz}.
\newblock {Sur certaines in\'egalit\'es dans la th\'eorie des fonctions avec
  quelques remarques sur les g\'eometries non-euclidiennes}.
\newblock {Kungl. Fysiogr. S\"allsk. Lund F\"orh. 1, 21.}, 1931.

\bibitem{zbMATH03029943}
M.~{Riesz}.
\newblock {Int\'egrales de Riemann-Liouville et potentiels}.
\newblock {\em {Acta Litt. Sci. Szeged}}, 9:1--42, 1938.

\bibitem{MR1485778}
E.~B. Saff and V.~Totik.
\newblock {\em Logarithmic potentials with external fields}, volume 316 of {\em
  Die Grundlehren der mathematischen Wissenschaften}.
\newblock Springer, 1997.
\newblock Appendix B by Thomas Bloom.

\bibitem{MR2090496}
M.~Schlosser.
\newblock Some curious extensions of the classical beta integral evaluation.
\newblock In {\em Mathematics and computer science. {III}}, Trends Math., pages
  59--68. Birkh\"{a}user, Basel, 2004.

\bibitem{MR0209834}
L.~Schwartz.
\newblock {\em Th\'{e}orie des distributions}.
\newblock Publications de l'Institut de Math\'{e}matique de l'Universit\'{e} de
  Strasbourg, IX-X. Hermann, Paris, 1966.
\newblock Nouvelle \'{e}dition, enti\'{e}rement corrig\'{e}e, refondue et
  augment\'{e}e.

\bibitem{MR0290095}
E.~M. Stein.
\newblock {\em Singular integrals and differentiability properties of
  functions}.
\newblock Princeton Mathematical Series, No. 30. Princeton University Press,
  Princeton, N.J., 1970.

\bibitem{MR2528466}
D.~W. Stroock.
\newblock Weyl's lemma, one of many.
\newblock In {\em Groups and analysis}, volume 354 of {\em London Math. Soc.
  Lecture Note Ser.}, pages 164--173. Cambridge Univ. Press, Cambridge, 2008.

\end{thebibliography}

\end{document}